\documentclass[a4paper, 10pt, twoside]{amsart}
\usepackage[top=3cm, bottom=3cm, left=2.5cm, right=2.5cm]{geometry}
\usepackage[pdftex]{graphicx}

\usepackage[utf8]{inputenc}
\usepackage[T1]{fontenc}

\usepackage[english]{babel}
\usepackage[noadjust]{cite}

\usepackage{amsthm,amssymb,epsfig,mathtools}
\usepackage{mathrsfs}  
\usepackage{bbm}        
\usepackage{esint}     
\usepackage{dsfont}
\usepackage{color}
\usepackage{enumerate}
\usepackage{url}
\usepackage{algorithm2e}

\usepackage{calc}
\usepackage[inline]{enumitem}
\usepackage[normalem]{ulem}
\usepackage{url}
\usepackage[bookmarks, bookmarksnumbered, pdfstartview={XYZ null null 1.00}]{hyperref}

\usepackage{tikz}
\usetikzlibrary{shapes, arrows, arrows.meta, decorations.markings}

\usepackage{todonotes}
\usepackage[ulem=normalem,draft]{changes}
\usepackage{cancel}

\makeindex

\numberwithin{equation}{section}

\newcommand{\theoname}{Theorem}
\newcommand{\lemmname}{Lemma}
\newcommand{\coroname}{Corollary}
\newcommand{\propname}{Proposition}
\newcommand{\definame}{Definition}

\newcommand{\remkname}{Remark}
\newcommand{\explname}{Example}

\theoremstyle{plain}
\newtheorem{theorem}{\theoname}[section]
\newtheorem{lemma}[theorem]{\lemmname}

\newtheorem{proposition}[theorem]{\propname}

\theoremstyle{definition}
\newtheorem{definition}[theorem]{\definame}
\newtheorem{remark}[theorem]{\remkname}

\newlist{hypothesis}{enumerate}{1}
\setlist[hypothesis]{label={\textup{(H\arabic*)}}, ref={(H\arabic*)}, leftmargin=*, widest*=10}

  {\list{}{\leftmargin=.5in\rightmargin=.5in}  \item[]  }%
  {\endlist}

\def\dd{{\rm d}}
\newcommand{\eqdef}{\ensuremath{\stackrel{\mbox{\upshape\tiny def.}}{=}}}

\newcommand{\norm}[1]{\left\lVert#1\right\rVert}
\newcommand{\inner}[1]{\left\langle#1\right\rangle}

\def\1B{{\bf  1}}

\def\dist{\mathop{\rm dist}}

\def\supp{\mathop{\rm supp}}
\def\id{{\mathop{\rm id}}}

\newcommand{\cvstar}[1]{\xrightharpoonup[#1]{\star}}

\newcommand{\cvstrong}[2]{\xrightarrow[#1]{#2}}

\def\inf{\mathop{\rm inf}}
\def\sup{\mathop{\rm sup}}

\def\min{\mathop{\rm min}}
\def\max{\mathop{\rm max}}

\def\argmin{\mathop{\rm argmin}}

\DeclareMathOperator{\proj}{proj}

\newcommand{\mres}{\mathbin{\vrule height 1.6ex depth 0pt width
		0.13ex\vrule height 0.13ex depth 0pt width 1.3ex}}

\newcommand{\Golab}{Go\l\k{a}b}
\def\H{\mathscr{H}}

\newcommand{\Sigmayrn}{\Sigma_{y_0,r_n}}

\title[]{Absence of loops for the Wasserstein-\texorpdfstring{$\H^1$}{H1} problem: the concentration/blow-up argument}

\author{João Miguel Machado}
\address{Lagrange Mathematical and Computing Center\\
103 rue de Grenelle\\
Paris, 75007}
\email{joao-miguel.machado@ceremade.dauphine.fr}


\date{\today}

\begin{document}

\begin{abstract}
  In the present work we prove that minimizers of the Wasserstein-$\mathscr{H}^1$ problem, introduced recently in~\cite{chambolle2025one}, are trees in three cases: when the target measure is a sum of finitely many Dirac masses, when it has a bounded density, when it is a mixture of an atomic measure and a measure with bounded density, {  and when it is a mixture of both these cases.} 
\end{abstract}
\maketitle
\tableofcontents

\section{Introduction}\label{section.introduction}
Consider the following problem: given a probability measure $\varrho_0 \in \mathscr{P}(\mathbb{R}^d)$, how can it best be approximated with a 1-dimensional set, that is how can we approximate it with a measure uniformly distributed over such lower dimensional sets? This question has been recently addressed with a variational approach in~\cite{chambolle2025one} with the following variational problem:\footnote{We recall that $\mu\mres A$ denotes the restriction measure of a Radon measure $\mu$ onto the set $A$, that is, for all Borel sets $B$, we have $\mu\mres A(B) \eqdef \mu(A\cap B)$.}
\begin{equation}\label{problem.shape_optimization}
	\tag{$P_\Lambda$}
	\inf_{\Sigma \text{ connected }} 
	W_p^p\left(
        \varrho_0, 
        \frac{1}{\H^1(\Sigma)}\H^1\mres \Sigma
    \right)
    + \Lambda \H^1(\Sigma),
\end{equation}
where $W_p$ corresponds to the Wasserstein distance, defined via the value of an optimal transport problem~\cite{ambrosio2021lectures,villani2009optimal,santambrogio2015optimal}, that metrizes the weak convergence of probability measures and $\H^1$ denotes the $1$-dimensional Hausdorff measure~\cite{ambrosio2000functions,maggi2012sets}. Notice that the penalization of the total length is necessary otherwise the Wasserstein distance could be made arbitrarily small by choosing a suitable space-filling curve, whereas without the connectedness constraint the same could be achieved by approximating $\varrho_0$ with a sequence of atomic measures, while having zero length.

In~\cite{chambolle2025one} existence of an optimal network $\Sigma$ has been proven, provided that the regularization parameter $\Lambda$ is small enough and that $\varrho_0$ does not give mass to $1$-dimensional sets. In addition, some qualitative properties of this problem have been studied. Still in~\cite{chambolle2025one} minimizers are shown to be Ahlfors regular; while in~\cite{machado2025phase} a phase-field approximation result for~\eqref{problem.shape_optimization} has been derived with an Ambrosio-Tortorelli type functional. The goal of this work is to show that optimal networks are trees, \textit{i.e.} none of its subsets is homeomorphic to $\mathbb{S}^1$. 

Differently from other similar problems, such as the Steiner~\cite{brazil2014history,paolini2013existence}, or the average distance minimizers problem~\cite{paolini2004qualitative,buttazzo2003optimal}, existence of an optimal network to~\eqref{problem.shape_optimization} does not follow directly from the \textit{Direct Method of the Calculus of Variations}. The difficulty stems from the lack of compatibility between the convergence of sets (Hausdorff convergence) and the narrow convergence of measures, see Section~\ref{section.preliminaries} for more details on such notions of convergence. Indeed, cluster points for sequences of the form $\H^1\mres \Sigma_n$ are not necessarily of the form $\H^1\mres\Sigma$ due to concentration of mass effects. 

For this reason, its lower semi-continuous relaxation is introduced, for which existence of minimizers can be easily shown with the direct method. This relaxation can be written as
\begin{equation}\label{problem.shape_optimization_relaxed}
	\tag{$\overline{P}_\Lambda$}
	\inf_{
		\substack{
			\nu \in \mathscr{P}_p(\mathbb{R}^d)
		}
	  }
	W_p^p(\varrho_0, \nu) + \Lambda \mathcal{L}(\nu)
\end{equation}
where the \textit{length functional} $\mathcal{L}$ is defined for a probability measure $\nu \in \mathscr{P}(\mathbb{R}^d)$ as 
\begin{equation}\label{eq.length_functional}
  \mathcal{L}(\nu)
    \eqdef
    \begin{cases}
      \min 
      \left\{
          \alpha \ge 0 : 
              \alpha \nu \ge \H^1\mres\supp\nu
      \right\},& \text{ if $\supp \nu$ is connected,}\\ 
      +\infty,& \text{ otherwise.}
    \end{cases}
\end{equation}
It can be proven that the quantity in~\eqref{eq.length_functional} coincides with the l.s.c.~relaxation of the functional defined by $\displaystyle \frac{1}{\H^1(\Sigma)} \H^1\mres \Sigma \mapsto \H^1(\Sigma)$, if $\Sigma$ is connected, and $+\infty$ otherwise. For more details and properties on the length functional, the reader is referred to~\cite{chambolle2025one} where it was first introduced, or to Section~\ref{sec.golab_length_tangent} for a brief discussion. 

With this new formulation of the problem, the proof of existence consists of showing that any minimizer of~\eqref{problem.shape_optimization_relaxed} is uniformly distributed over its support, being therefore a solution to~\eqref{problem.shape_optimization}. Heuristically this can be easily done once one has proven the following property: suppose that $\nu$ is a minimizer of~\eqref{problem.shape_optimization_relaxed}, if it has an excess, that is regions where its density is not constant, it can be proved that this excess to the uniform density is formed through projections onto $\Sigma$, in other words the optimal transportation plan from $\varrho_0$ onto the minimizer is such that the mass of the excess is at minimal distance from $\Sigma$, see step (1) below in this section or the thesis of Proposition~\ref{prop.projection_to_loops}. Therefore, in principle one could construct a better competitor with a constant density by replacing any excess of the uniform density with segments in the opposite direction of the projections, as represented in Figure~\ref{figure.sketch_existence}. 
\begin{figure}[h!tbp]
  \centering
  \tikzset{every picture/.style={line width=0.7pt}} 

\begin{tikzpicture}[x=0.65pt,y=0.7pt,yscale=-1,xscale=1]

\draw   (150.38,45.6) -- (634.1,45.6) -- (500.82,277.6) -- (17.1,277.6) -- cycle ;
\draw [color={rgb, 255:red, 74; green, 144; blue, 226 }  ,draw opacity=1 ][line width=1.5]    (62.05,258.15) .. controls (76.07,274.76) and (140.41,245.42) .. (175.4,177.28) .. controls (210.39,109.13) and (173.22,131) .. (191.43,116.39) .. controls (209.64,101.78) and (198.62,121.22) .. (177.66,105.17) .. controls (156.7,89.12) and (144.9,92.6) .. (131.77,99.58) ;
\draw [color={rgb, 255:red, 74; green, 144; blue, 226 }  ,draw opacity=1 ][line width=1.5]    (172.79,102.74) .. controls (182.79,78.11) and (154.13,120.93) .. (176.51,68.87) ;
\draw [color={rgb, 255:red, 74; green, 144; blue, 226 }  ,draw opacity=1 ][line width=1.5]    (70.77,230.57) .. controls (145.76,233.9) and (39.93,226.14) .. (101.43,254.95) ;
\draw  [dash pattern={on 0.84pt off 2.51pt}]  (175.4,177.28) -- (270.8,165.9) ;
\draw  [dash pattern={on 0.84pt off 2.51pt}]  (153.59,211.44) -- (238.45,221.05) ;
\draw  [fill={rgb, 255:red, 208; green, 2; blue, 27 }  ,fill opacity=0.79 ] (240.54,175.11) .. controls (256.98,160.46) and (299.4,169.7) .. (299.69,175.11) .. controls (299.98,180.53) and (256.76,188.18) .. (259.41,207.98) .. controls (262.07,227.78) and (228.39,220.03) .. (215.79,215.9) .. controls (203.19,211.78) and (245.17,201.38) .. (250.02,193.46) .. controls (254.87,185.54) and (224.11,189.77) .. (240.54,175.11) -- cycle ;
\draw  [dash pattern={on 0.84pt off 2.51pt}]  (198.59,111.08) -- (252.97,64.68) ;
\draw  [dash pattern={on 0.84pt off 2.51pt}]  (201.73,111.08) -- (214.92,110.96) -- (269.34,110.45) ;
\draw  [fill={rgb, 255:red, 208; green, 2; blue, 27 }  ,fill opacity=0.72 ] (226.54,90.81) .. controls (232.37,79.78) and (246,70.85) .. (256.98,70.85) .. controls (267.95,70.85) and (272.11,79.78) .. (266.28,90.81) .. controls (260.44,101.83) and (246.82,110.76) .. (235.84,110.76) .. controls (224.87,110.76) and (220.7,101.83) .. (226.54,90.81) -- cycle ;
\draw  [fill={rgb, 255:red, 208; green, 2; blue, 27 }  ,fill opacity=1 ] (195.04,111.08) .. controls (194.96,109.2) and (196.47,107.68) .. (198.43,107.68) .. controls (200.39,107.68) and (202.05,109.2) .. (202.14,111.08) .. controls (202.23,112.95) and (200.71,114.48) .. (198.75,114.48) .. controls (196.79,114.48) and (195.13,112.95) .. (195.04,111.08) -- cycle ;
\draw [color={rgb, 255:red, 74; green, 144; blue, 226 }  ,draw opacity=1 ][line width=1.5]    (326.74,260.97) .. controls (339.17,277.83) and (400.62,248.06) .. (436.21,178.92) .. controls (471.79,109.78) and (436.09,131.96) .. (453.75,117.14) .. controls (471.4,102.32) and (460.27,122.05) .. (441.32,105.76) .. controls (422.36,89.47) and (411.17,93.01) .. (398.59,100.09) ;
\draw [color={rgb, 255:red, 74; green, 144; blue, 226 }  ,draw opacity=1 ][line width=1.5]    (436.86,103.29) .. controls (447.24,78.3) and (418.63,121.75) .. (441.75,68.93) ;
\draw [color={rgb, 255:red, 74; green, 144; blue, 226 }  ,draw opacity=1 ][line width=1.5]    (336.05,232.99) .. controls (406.11,236.36) and (307.36,228.5) .. (363.73,257.72) ;
\draw  [dash pattern={on 0.84pt off 2.51pt}]  (436.21,178.92) -- (525.98,167.37) ;
\draw  [dash pattern={on 0.84pt off 2.51pt}]  (414.37,213.57) -- (493.42,223.33) ;
\draw  [fill={rgb, 255:red, 208; green, 2; blue, 27 }  ,fill opacity=0.79 ] (497.28,176.72) .. controls (513.27,161.86) and (552.6,171.23) .. (552.64,176.72) .. controls (552.69,182.21) and (511.91,189.98) .. (513.58,210.07) .. controls (515.24,230.16) and (484.04,222.29) .. (472.41,218.11) .. controls (460.79,213.92) and (500.52,203.38) .. (505.39,195.34) .. controls (510.26,187.3) and (481.29,191.59) .. (497.28,176.72) -- cycle ;
\draw  [dash pattern={on 0.84pt off 2.51pt}]  (460.67,111.75) -- (513.5,64.68) ;
\draw  [dash pattern={on 0.84pt off 2.51pt}]  (463.6,111.75) -- (475.96,111.63) -- (526.93,111.12) ;
\draw  [fill={rgb, 255:red, 208; green, 2; blue, 27 }  ,fill opacity=0.72 ] (487.67,91.19) .. controls (493.59,80) and (506.72,70.94) .. (516.99,70.94) .. controls (527.27,70.94) and (530.79,80) .. (524.87,91.19) .. controls (518.95,102.37) and (505.82,111.44) .. (495.55,111.44) .. controls (485.28,111.44) and (481.75,102.37) .. (487.67,91.19) -- cycle ;
\draw [color={rgb, 255:red, 74; green, 102; blue, 226 }  ,draw opacity=1 ][line width=1.5]    (460.67,111.75) -- (481.25,102.16) ;
\draw [color={rgb, 255:red, 74; green, 102; blue, 226 }  ,draw opacity=1 ][line width=1.5]    (432.43,186.55) -- (447.1,187.6) ;
\draw [color={rgb, 255:red, 74; green, 102; blue, 226 }  ,draw opacity=1 ][line width=1.5]    (426.61,195.7) -- (440.1,197.6) ;
\draw [color={rgb, 255:red, 74; green, 102; blue, 226 }  ,draw opacity=1 ][line width=1.5]    (419.69,205.25) -- (433.1,206.6) ;
\draw    (316.04,161.38) .. controls (325.38,146.27) and (356.19,137.24) .. (389.8,150.16) ;
\draw    (316.04,161.38) .. controls (341.33,145.75) and (366.3,152.95) .. (375.59,165.74) ;
\draw    (375.59,165.74) -- (362.69,179.66) ;
\draw    (389.8,150.16) -- (399.76,139.24) ;
\draw    (399.76,139.24) -- (402.24,174.56) ;
\draw    (362.69,179.66) -- (402.24,174.56) ;
\draw    (323.82,92.73) -- (278.35,80.63) ;
\draw [shift={(276.42,80.11)}, rotate = 14.9] [color={rgb, 255:red, 0; green, 0; blue, 0 }  ][line width=0.75]    (10.93,-3.29) .. controls (6.95,-1.4) and (3.31,-0.3) .. (0,0) .. controls (3.31,0.3) and (6.95,1.4) .. (10.93,3.29)   ;
\draw    (332.6,105.34) -- (297,162.27) ;
\draw [shift={(295.94,163.96)}, rotate = 302.02] [color={rgb, 255:red, 0; green, 0; blue, 0 }  ][line width=0.75]    (10.93,-3.29) .. controls (6.95,-1.4) and (3.31,-0.3) .. (0,0) .. controls (3.31,0.3) and (6.95,1.4) .. (10.93,3.29)   ;
\draw    (150.38,45.6) -- (9.1,290.87) ;
\draw [shift={(8.1,292.6)}, rotate = 299.94] [color={rgb, 255:red, 0; green, 0; blue, 0 }  ][line width=0.75]    (10.93,-3.29) .. controls (6.95,-1.4) and (3.31,-0.3) .. (0,0) .. controls (3.31,0.3) and (6.95,1.4) .. (10.93,3.29)   ;
\draw    (150.38,45.6) -- (647.1,45.6) ;
\draw [shift={(649.1,45.6)}, rotate = 180] [color={rgb, 255:red, 0; green, 0; blue, 0 }  ][line width=0.75]    (10.93,-3.29) .. controls (6.95,-1.4) and (3.31,-0.3) .. (0,0) .. controls (3.31,0.3) and (6.95,1.4) .. (10.93,3.29)   ;
\draw    (150.38,45.6) -- (150.12,12.6) ;
\draw [shift={(150.1,10.6)}, rotate = 89.54] [color={rgb, 255:red, 0; green, 0; blue, 0 }  ][line width=0.75]    (10.93,-3.29) .. controls (6.95,-1.4) and (3.31,-0.3) .. (0,0) .. controls (3.31,0.3) and (6.95,1.4) .. (10.93,3.29)   ;
\draw [color={rgb, 255:red, 208; green, 2; blue, 27 }  ,draw opacity=0.31 ][line width=1.5]    (198.59,111.08) -- (198.31,76.08) ;
\draw [shift={(198.31,76.08)}, rotate = 269.54] [color={rgb, 255:red, 208; green, 2; blue, 27 }  ,draw opacity=0.31 ][fill={rgb, 255:red, 208; green, 2; blue, 27 }  ,fill opacity=0.31 ][line width=1.5]      (0, 0) circle [x radius= 4.36, y radius= 4.36]   ;
\draw    (128.1,140.6) -- (183.27,116.4) ;
\draw [shift={(185.1,115.6)}, rotate = 156.32] [color={rgb, 255:red, 0; green, 0; blue, 0 }  ][line width=0.75]    (10.93,-3.29) .. controls (6.95,-1.4) and (3.31,-0.3) .. (0,0) .. controls (3.31,0.3) and (6.95,1.4) .. (10.93,3.29)   ;
\draw    (125.1,159.6) -- (155.55,184.34) ;
\draw [shift={(157.1,185.6)}, rotate = 219.09] [color={rgb, 255:red, 0; green, 0; blue, 0 }  ][line width=0.75]    (10.93,-3.29) .. controls (6.95,-1.4) and (3.31,-0.3) .. (0,0) .. controls (3.31,0.3) and (6.95,1.4) .. (10.93,3.29)   ;
\draw [color={rgb, 255:red, 208; green, 2; blue, 27 }  ,draw opacity=1 ][fill={rgb, 255:red, 208; green, 2; blue, 27 }  ,fill opacity=0.24 ]   (175.4,177.28) .. controls (167.23,167.13) and (166.1,155.6) .. (163.1,159.6) .. controls (160.1,163.6) and (164.81,179.27) .. (164.33,185.58) .. controls (163.84,191.89) and (161.85,170.6) .. (154.1,170.6) .. controls (146.35,170.6) and (155.1,186.6) .. (153.59,211.44) ;

\draw (109,139.92) node [anchor=north west][inner sep=0.75pt]  [font=\footnotesize]  {$\nu _{\text{exc}}$};
\draw (332.82,85.09) node [anchor=north west][inner sep=0.75pt]  [font=\footnotesize]  {$\varrho _{\text{exc}}$};

\end{tikzpicture}
  \caption{Heuristic proof of existence of an optimal shape for problem~\eqref{problem.shape_optimization}. If a solution has an excess part, represented in the figure by a measure having a density along $\Sigma$ and a Dirac mass, it must be formed through projections onto $\Sigma$.  But then it is better to send the excess mass that is being projected to small segments in the direction of the projection.}\label{figure.sketch_existence}
\end{figure}
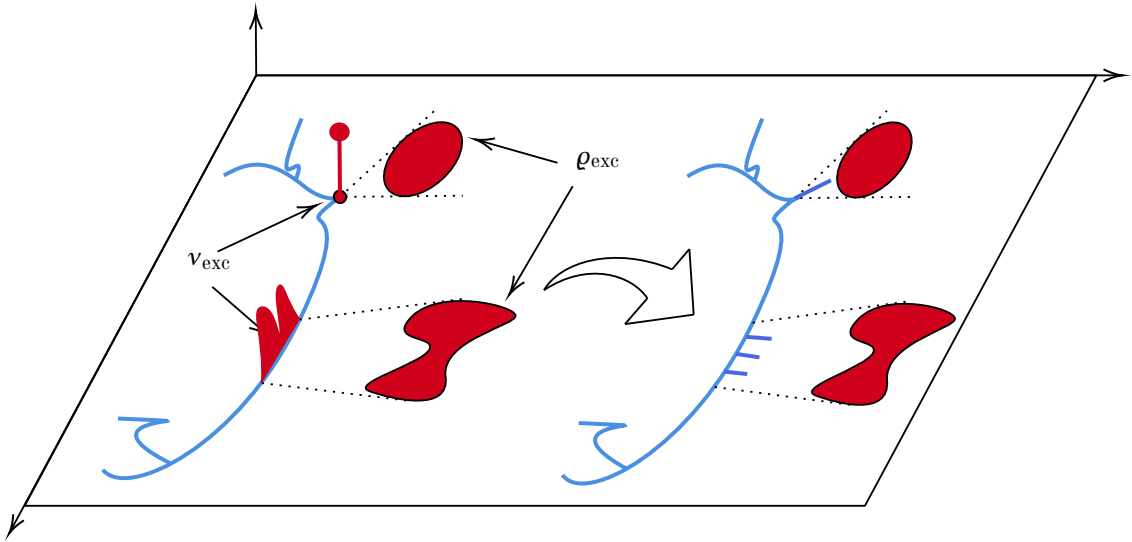

However, since we lack much information on the measure that is projected to form the excess, it is unclear \textit{a priori} how to select to which directions should point the segments that decrease the energy. For this reason, in~\cite{chambolle2025one} a \textit{concentration/blow-up} argument is developed, that yields a localized problem which inherits the projection property. In the blow-up limit, the optimal network $\Sigma$ is replaced by its \textit{approximate tangent space} $T_{y_0}\Sigma$ (see Section~\ref{sec.golab_length_tangent}) at a carefully chosen point $y_0$. This simplifies the construction of a better competitor since now all projection directions are orthogonal to $T_{y_0}\Sigma$.

In principle, the concentration/blow-up argument can be carried out for any structure that is formed via projections onto the optimal network. As a result, if we can prove that loops are formed through projections, one could also expect that optimal networks should not have them, with a similar heuristic from the question of existence. 

For this reason, the first result that we establish in this work is that if a loop exists, it must be formed via projections, hence one can localize around a carefully chosen point and ``open'' the loop, while adding a structure that reduces the cost of projecting onto $\Sigma$, see Figure~\ref{figure.sketch_noloops}. Once again, conducting this argument directly is not simple since we cannot control the direction of projection onto the loop, therefore we implement a variation of the \textit{concentration/blow-up argument} that is described in more detail in the sequel.  

\begin{figure}[h!tbp]
  \centering
  \input{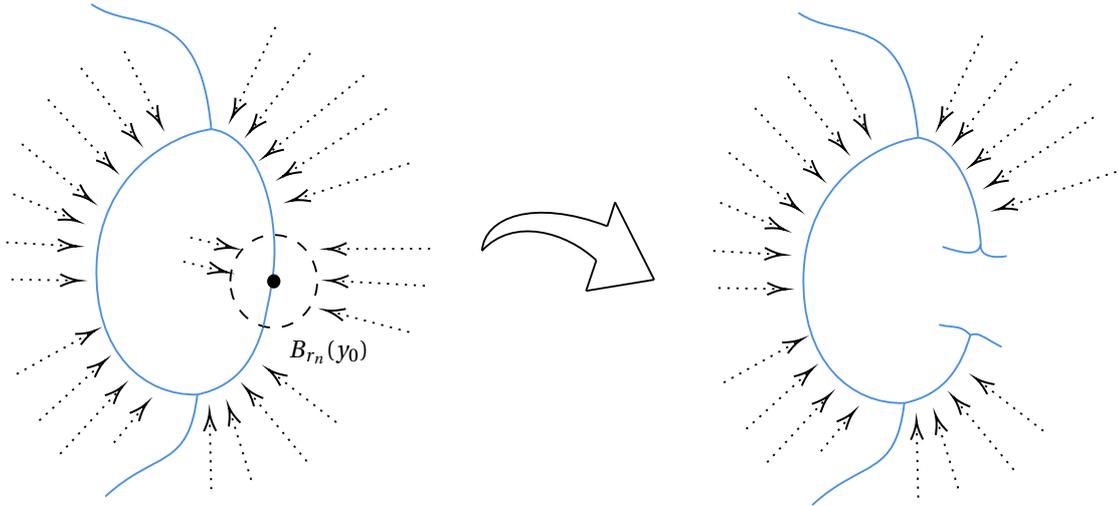}
  \caption{Argument for absence of loops for~\eqref{problem.shape_optimization_relaxed}. As in the proof of existence, we begin by showing that loops are formed through projections and later use this information to construct a better competitor.}\label{figure.sketch_noloops}
\end{figure}

\subsection{Contributions and the concentration/blow-up argument}\label{subsection.localization_blowup_arg}
As previously stated, in this work we show that the support of  minimizers of~\eqref{problem.shape_optimization_relaxed} are trees in three cases
\begin{enumerate}
  \item[\underline{Case 1:}] if $\varrho_0$ is a convex combination of Dirac masses, \textit{i.e.}
  \[
      \varrho_0 = \mu_N = \sum_{i=1}^N a_i\delta_{x_i}, \text{ for $0 \le ai$ and }
      \sum_{i=1}^N a_i = 1. 
  \]
  \item[\underline{Case 2:}] $\varrho_0$ is absolutely continuous w.r.t.~the Lebesgue measure with compact support and bounded density, \textit{i.e.} $\varrho_0 \in L^\infty(\mathbb{R}^d)$.
  \item[\underline{Case 3:}]{  $\varrho_0$ is  of the form 
  \[
    \varrho_0 = \varrho + \sum_{i=1}^N a_i\delta_{x_i}, \text{ for $0 \le ai$ and }
      \sum_{i=1}^N a_i < 1,
  \]
  and $\varrho \in L^\infty(\mathbb{R}^d)$ with compact support.}
\end{enumerate}

{  
We then prove the following result. 
\begin{theorem}\label{theorem.intro}
    Assume that $\varrho_0$ is an atomic measure, as in Case 1. Then the support $\Sigma$ of any solution to the relaxed problem~\eqref{problem.shape_optimization} is a tree, in the sense that it does not contain homeomorphic images of $\mathbb{S}^1$.
    
    In the Case 2, that is $\varrho_0 \in L^\infty(\Omega)$ is given by a bounded density with compact support, assume that either $d\ge 3$ and $p \ge 1$, or that $d = 2$ and $p>1$. Then the original problem~\eqref{problem.shape_optimization} admits solutions $\Sigma$, and any such optimal network is a tree.

    In Case 3, where $\varrho_0$ is a mixture between a bounded density and an atomic measure, if either $d\ge 3$ and $p \ge 1$ or if $d = 2$ and $p>1$, the support of any solution to the relaxed problem is a tree. 
\end{theorem}
}

To prove Theorem~\ref{theorem.intro} we apply the concentration/blow-up argument, also used in~\cite{chambolle2025one} for the existence of optimal networks to~\eqref{problem.shape_optimization}. More generally, if the connectedness constraints can be preserved, it could be used to rule out the appearance of any structure that is formed through projections. Hopefully this strategy of proof could be useful in other contexts, so in the sequel we go through each step. 

\begin{enumerate}
  \item \textbf{\underline{Identify a structure that is formed through projections}:} In the first step one proves that the structure one wishes to exclude is formed via projections of the initial measure $\varrho_0$ using the optimal transport problem in the energy from~\eqref{problem.shape_optimization_relaxed}. More precisely for a minimizer $\nu_\star$ with support $\Sigma$, {   this means identifying some $\Gamma \subset \Sigma$ for which, if $\gamma$ is an optimal transportation plan from $\varrho_0$ to $\nu_\star$, then 
  \[
        |x-y| = \dist(x,\Sigma) \text{ for $\gamma$-a.e. $(x,y)$.}
  \]}
  Such structures can be loops or the excess measure, mentioned above for the proof of existence from~\cite{chambolle2025one}. For more details on what we mean by a loop being formed via projections, see the thesis of Proposition~\ref{prop.projection_to_loops}.  
  \item \textbf{\underline{Chose a point $y_0$ with good properties to localize}:} The next step is to select a point $y_0$ from this structure (inside the loop, or on the support of the excess measure) for which we can make variations, for instance such that the approximate tangent space $T_{y_0}\Sigma$ exists, and that is a non-cut point for the absence of loops, allowing to remove a neighborhood of it without breaking the connectedness.
  \item \textbf{\underline{Define localized problems and show they $\Gamma$-converge}:} 
  
  With the previous step, we obtain a natural family of functionals ${\left(F_n\right)}_{n \in \mathbb{N}}$, which is minimized by a family of localizations of the optimal measures to the original problem, supported over a ball $B_{r_n}(y_0)$. Since the mass of such localizations vanishes as $r_n \to 0$, we rescale them with a factor $r_n^{-1}$, bringing them back to the unitary ball $B_1$. The passage to the limit as $r_n \to 0$ of such sequence of rescaled localizations is called a \textit{blow-up}. The issue is that in this process, we must also rescale the mass of a portion of the original measure. Since the rescaled versions of the original measure might lose mass at infinity, we first concentrate it inside a ball whose radius is of order $r$. This \textit{concentration step} preserves the total mass of the limit problem which consists of minimizing a limit functional $F$ of the sequence $F_n$ in the sense of $\Gamma$-convergence. 
  
  In the proof of existence, it is necessary that the variations satisfy the density penalization introduced by the length functional~\eqref{eq.length_functional}. In the case of the absence of loops, we must be careful with the connectedness constraint, hence the ball $B_{r_n}(y_0)$ should be chosen so that $\Sigma\setminus B_{r_n}(y_0)$ remains connected.

  \item \textbf{\underline{Show that the projection property passes to the limit}:} In this step, we use the fundamental property of $\Gamma$ convergence, so that the sequence of localizations that minimize the functionals $F_n$ converge to a minimizer of the limit $F$. In addition, we also verify that the projection property proved in step (1) also holds in the limit problem, so that this minimizer of $F$ is also formed via projections, but this time onto the approximate tangent space $T_{y_0}\Sigma$.
  
  \item \textbf{\underline{Construct a better competitor for $F$}:} Finally, we exploit the projection property of the limit to construct a strictly better competitor for the minimization of $F$. This contradicts the entire construction, and in particular contradicts the existence of the structure from step (1).
\end{enumerate}

This argument is reminiscent of an approach from Santambrogio and Tilli in~\cite{santambrogio2005blow} used to fully characterize the blow-ups of any point from optimal networks for the \textit{average distance functional}, see~\cite{lemenant2012presentation}. In their work, a crucial ingredient was the full topological characterization of such optimal networks done in $\mathbb{R}^2$ since the introduction of the problem by Butazzo and Stepanov in~\cite{buttazzo2003optimal}, where it was proven that optimizers are trees with finitely many branching points, each one being triple junctions of 120 degrees. 

This result has recently been generalized to $\mathbb{R}^d$ in~\cite{obrien2025structure}. Their approach consists of defining a vector field, the \textit{barycenter field}, which measures from which direction the mass is on average being projected onto the network. This allows them to develop a local improvement theory of the average distance problem. Adapting these techniques to the Wasserstein-$\mathscr{H}^1$ problem might be an interesting direction of investigation, which can hopefully shed some light onto other topological properties of minimizers for our problem. 

\subsection{Structure of this manuscript}

In Section~\ref{section.preliminaries} we make a brief review of the basic facts of optimal transport and geometric measure theory, which shall be useful for our analysis. A particular emphasis is given to Section~\ref{sec.preliminaries_loops}, where we study a slight refinement of a classical lemma used to prove absence of loops in problems such as the Steiner or the average distance problems. 

This refinement might not be surprising to seasoned experts on the field, but is particularly relevant to the implementation of the concentration/blow-up argument, which is done in Section~\ref{sec.absense_loops} and culminates at Theorem~\ref{theorem.no_loops_integrable_regime} where we obtain the desired absence of loops. Some proofs therein are postponed to Appendix~\ref{appendix}, since they are only minor variations of the proofs from~\cite{chambolle2025one}. 

\subsection*{Acknowledgments}
The author thanks Antonin Chambolle, Vincent Duval and Forest Kobayashi for many discussions which lead to an improved version of the present paper. This work was also greatly improved due to the detailled comments of the annonimous referrees, to whom the authour is greatly thankful. The authour aussi wishes to acknowkedge the support of the Lagrange Mathematics and Computing Research Center.

\section{Preliminaries}\label{section.preliminaries}
In this section we recall the notions of convergence of sets and measures required in this article as well as the tools from geometric measure theory that will be employed. Most of the results presented here are well known and are recalled for the sake of readability, as well as to establish notation. Therefore, more experienced readers may want to skip this, except maybe for Lemma~\ref{lemma.noncut_property} from subsection~\ref{sec.preliminaries_loops}, which is a small refinement of a result frequently used in the literature to prove absence of loops in 1-dimensional shape optimization problems, see for instance~\cite[Lemma 6.1]{buttazzo2003optimal}. The usual result says that around every non-cut point one can remove a connected set with diameter as small as we want and still keep the connectedness of the network. This improvement says that such sets can be taken to be the intersection of the network and balls of arbitrarily small radius around the non-cut point, which is very convenient to perform the concentration/blow-up argument in the sequel. 

\subsection{Convergence of sets and measures}
To formulate variational problems on the space of connected sets, it is essential to equip this space with a topology that preserves connectedness and finite length. For this, \textit{Hausdorff} and \textit{Kuratowski} convergences are introduced, as detailed in~\cite{rockafellar2009variational}. These convergences are shown to maintain the desired properties when restricted to connected sets with bounded length.

\begin{definition}\label{definition.Hausdoff_Kuratowski_convergence}
	Let ${\left(A_n\right)}_{n \in \mathbb{N}}$ be a sequence of closed sets of $\mathbb{R}^d$. If $A \subset \mathbb{R}^d$ is closed, we say that 
	\begin{itemize}
		\item $A_n$ converges in the Hausdorff sense to $A$ if $d_H(A_n, A) \xrightarrow[n \to \infty]{} 0$, where $d_H$ is called the Hausdorff distance and is defined as 
		\begin{equation}
			\label{hausdord.distance}
			d_H(A,B) \eqdef \max\left\{
				\sup_{a \in A} \dist(a, B), \sup_{b \in B} \dist(b, A)
			\right\}, 
			\text{ we write $A_n \xrightarrow[n \to \infty]{d_H} A$. }
		\end{equation}
		\item A sequence of closed sets $C_n$ converges in the sense of Kuratowski to $C$, and we write $C_n \xrightarrow[n \to \infty]{K} C$, when
		\begin{enumerate}
			\item for all sequences $x_n \in C_n$, all its cluster points are contained in $C$.
			\item For all points $x\in C$ there exists a sequence $x_n \in C_n$, converging to $x$.
		\end{enumerate}
	\end{itemize}
\end{definition}

Furthermore, $A_n \xrightarrow[n \to \infty]{d_H} A$ if and only if $\dist(\cdot, A_n)\xrightarrow[n \to \infty]{} \dist(\cdot, A)$ uniformly. Similarly, Kuratowski convergence corresponds to the agreement of inner and outer limits: 
\begin{align*}
	\liminf_{n \to \infty} C_n &\eqdef 
    \left\{
		x \in \mathbb{R}^d : \limsup_{n \to \infty} \dist(x, C_n) = 0
	\right\}, 
	\\ 
	\limsup_{n \to \infty} C_n &\eqdef 
    \left\{
		x \in \mathbb{R}^d : \liminf_{n \to \infty} \dist(x, C_n) = 0
	\right\},
\end{align*}
in other words Kuratowski convergence holds if and only if $\dist(\cdot, C_n) \to \dist(\cdot, C)$ pointwise. Since the distance functions are 1-Lipschitz, by Ascoli-Arzelà’s Theorem we have that
\[ 
	\text{ 
	$C_n \xrightarrow[n \to \infty]{K} C$ if and only if $\dist(\cdot, C_n) \xrightarrow[n \to \infty]{} \dist(\cdot, C)$ locally uniformly.}
\]
As a result, Hausdorff convergence implies Kuratowski convergence, and both notions coincide on compact sets. Importantly, \textit{Blaschke's Theorem}, see~\cite[Theorem~6.1]{ambrosio2000functions}, states that the Hausdorff topology inherits compactness from the compactness of uniform convergence of the distance functions.

\subsection{Narrow convergence of probability measures and the Wasserstein distances}
Due to Riesz' representation theorem the set of Radon measures $\mathscr{M}(\mathbb{R}^d)$ is known to be the topological dual of the continuous functions. As a result, it is frequently endowed with the \textit{local weak-$\star$} convergence: a sequence ${\left(\mu_n\right)}_{n \in \mathbb{N}}$ is said to converge narrowly to $\mu$ in $\mathscr{M}_{loc}(\mathbb{R}^d)$~\cite[Def.~1.58]{ambrosio2000functions} if  
\[
    \int_{\mathbb{R}^d}\phi \dd \mu_n
    \cvstrong{n \to \infty}{}
    \int_{\mathbb{R}^d}\phi \dd \mu
    \text{ for all 
    $\phi \in \mathscr{C}_c(\mathbb{R}^d)$.
    }
\]
This notion of convergence however does not preserve the total mass of the sequence ${\left(\mu_n\right)}_{n \in \mathbb{N}}$, as a portion of the mass can be lost at infinity.  This is one of the difficulties in implementing Step (4) of the concentration/blow-up argument, see the discussion before Lemma~\ref{lemma.localization}. 

For this reason, when working with Radon probability measures it is customary to work with the \textit{narrow topology}, defined by replacing the space of continuous functions with compact support $\mathscr{C}_c(\mathbb{R}^d)$ by the class of continuous and bounded functions $\mathscr{C}_b(\mathbb{R}^d)$. Naturally, if the supports of a convergent sequence ${\left(\mu_n\right)}_{n \in \mathbb{N}}$ are all contained in the same compact subset of $\mathbb{R}^d$, then both notions of convergence coincide and the mass is preserved even under the weak-$\star$ convergence. This will be the case most times in this work, unless when we deal with blow-ups of sets and measures, when it is inevitable to send the support of the measures to infinity. 

Nonetheless, the narrow topology is actually metrizable and a possible choice of distance for this topology are the so called \textit{$p$-Wasserstein distances}\footnote{To be more precise, convergence with respect to the $p$-Wasserstein distance is equivalent to narrow convergence plus convergence of the $p$-moments, but the second condition is trivial in compact domains, which will be always the case where this is exploited in this paper.} defined via the value of an optimal transportation problem(see~\cite{ambrosio2021lectures,santambrogio2015optimal,villani2009optimal} for more details) as follows: given $\mu,\nu \in \mathscr{P}(\mathbb{R}^d)$ with finite $p$-moments, $p\ge 1$, the $p$-Wasserstein distance is defined as 
\[
    W_p^p(\mu,\nu) 
    \eqdef 
    \min_{\gamma \in \Pi(\mu,\nu)} 
    \int_{\mathbb{R}^d\times \mathbb{R}^d} 
    |x-y|^p\dd \gamma(x,y),
\]
where $\displaystyle \Pi(\mu,\nu) \eqdef 
\left\{
    \gamma \in \mathscr{P}(\mathbb{R}^d\times \mathbb{R}^d):
    {(\pi_0)}_\sharp \gamma = \mu, \ {(\pi_1)}_\sharp \gamma = \nu
\right\}$ corresponds to the couplings with marginals $\mu$ and $\nu$. This corresponds to Kantorovitch's formulation of the problem, which is known under certain conditions to actually be a solution to Monge's problem 
\[
    \inf_{T_\sharp\mu = \nu} 
    \int_{\mathbb{R}^d} 
    |x - T(x)|^p\dd \mu(x),
\]
where the \textit{pushforward measure} is $T_\sharp \mu (A) \eqdef \mu(T^{-1}(A))$, for any Borel set $A\subset \mathbb{R}^d$. The connection between both formulations is give by Brenier's Theorem which states that whenever $\mu$ does not give mass to $(d-1)$-dimensional sets, there is a unique optimal transportation plan that is actually induced by a map, it can be written as $\gamma = {(\id, T)}_\sharp \mu$. 

\subsection{\Golab's Theorem, the length functional, blow-ups and approximate tangent spaces}\label{sec.golab_length_tangent}
In the sequel, we consider a sequence of compact and connected sets ${\left(\Sigma_n\right)}_{n \in \mathbb{N}}$ converging to $\Sigma$ in the sense of Kuratowski. We are mostly interested in the sequence of measures $\H^1\mres \Sigma_n$, up to subsequences, we can always assume it to converge weakly to a measure $\mu$. The classical version of \Golab's Theorem says that $\mu \ge \H^1\mres \Sigma$, while in~\cite{chambolle2025one}, this result is proved under the weaker Kuratowski convergence and the sequence $\Sigma_n$ doesn't have to be bounded, in fact it can have infinite length, as long as it is locally finite. 
\begin{theorem}[Density version of \Golab's Theorem]\label{theorem.Golab_localversion}
	Let ${(\Sigma_n)}_{n \in \mathbb{N}}$ be a sequence of closed and connected subsets of $\mathbb{R}^d$ converging in the sense of Kuratowski to some closed set $\Sigma$ and having locally uniform finite length, {\em i.e.} for all $R > 0$
	\[
		\sup_{n \in \mathbb{N}} \H^1(\Sigma_n \cap B_R) < +\infty.
	\]
	Define the measures $\mu_n \eqdef \H^1\mres \Sigma_n$, and let $\mu$ be a weak-$\star$ cluster point of this sequence. Then supp$\mu \subset \Sigma$ and it holds that 
	\[
		\mu \ge \H^1\mres \Sigma,
	\]
	in the sense of measures. 
\end{theorem}

This result is central to understand the \textit{length functional} described in the introduction. Consider the functional defined over the space of probability measures as 
\begin{equation}
    \ell(\nu) 
    \eqdef 
    \begin{cases}
        \mathscr{H}^1(\Sigma),& 
        \displaystyle
        \text{ if } \nu = \frac{1}{\mathscr{H}^1(\Sigma)}\mathscr{H}^1\mres \Sigma, \text{ for $\Sigma$ connected},\\ 
        +\infty,& \text{ otherwise.} 
    \end{cases}
\end{equation} 
Using \Golab's Theorem, one can show that the lower semi-continuous relaxation of the above functional is given by the length functional 
\begin{equation}
    \mathcal{L}(\nu) 
    \eqdef 
    \begin{cases}
        \min
        \left\{
            \alpha \ge 0 : \alpha \nu \ge \mathscr{H}^1\mres \supp \nu
        \right\},& 
        \displaystyle
        \text{ if } \supp \nu  \text{ is connected},\\ 
        +\infty,& \text{ otherwise,} 
    \end{cases}
\end{equation}
which is used in the definition of the relaxed formulation~\eqref{problem.shape_optimization_relaxed} and allows for much more flexibility once creating competitors to optimizers and extract information from them, as for instance in the proof of Proposition~\ref{prop.projection_to_loops}. The challenge associated with this functional is that, as opposed with $\Sigma \mapsto \mathscr{H}^1(\Sigma)$ it has a non-local flavor. Indeed, if we want to \textit{reduce the value} of $\mathcal{L}(\nu)$ we must \textit{increase the $\mathscr{H}^1$ density} of $\nu$ along all of its support $\Sigma$, even if we just want to study the behavior of a small neighborhood of a point in $\Sigma$. This is particularly inconvenient when combined with an optimal transportation cost. On the other hand, adding any structure to $\Sigma$, with a smaller density will increase the value of $\mathcal{L}$.

\Golab 's Theorem is also useful to extract a finer information on the blow-ups of $1$-rectifiable connected sets. Due to a result from Besicovitch, we know that the connected sets $\Sigma$ with finite length that are of interest to us are actually countably $\H^1$-rectifiable~\cite{ambrosio2000functions,maggi2012sets}. In other words, up to $\H^1$-negligible sets they can be written as the countable union of Lipschitz images, that is there are Lipschitz functions $f_i:[0,1]\mapsto \mathbb{R}^d$ such that
\[
    \H^1\left(
        \Sigma \setminus \bigcup_{i \in \mathbb{N}} f_i([0,1]) 
    \right) = 0.
\]

As such, this class of sets enjoy tangentiability properties $\H^1$ almost everywhere, see for instance~\cite{delellis2006lecture,maggi2012sets}. In other words, we know from the so called \textit{blow-up Theorem} (\cite[Theorem~10.2]{maggi2012sets}) that for \textit{a.e.} $x\in\Sigma$, it holds that
\begin{equation}\label{eq.blowup_measure}
    \frac{1}{r}
    {\left(\Phi^{x,r}\right)}_\sharp 
    \H^1\mres \Sigma 
    =  
    \H^1\mres\left(
        \frac{\Sigma - x}{r} 
    \right)    
    \cvstar{r \to 0}
    \H^1\mres T_x\Sigma, 
    \text{ where }
    \Phi^{x,r} \eqdef \frac{\id - x}{r},
\end{equation}
and $T_x\Sigma$ is a one-dimensional subspace of $\mathbb{R}^d$, which is called the \textit{approximate tangent space of $\Sigma$ at $x$}. {  As a direct consequence, defining the $1$-density of a Radon measure $\mu$ as the quantity 
\begin{equation}\label{eq.1-density}
    \theta_1(\mu,x) 
    \eqdef 
    \lim_{r \to 0^+} 
    \frac{\mu(B_r(x))}{2r}
    \text{ and }
    \theta_1(\Sigma,x) = \theta_1(\mathscr{H}^1\mres \Sigma,x),
\end{equation}
for $\mathscr{H}^1$-a.e.~point $x$ of a $1$-rectifiable set $\Sigma$ it holds that $\theta_1(\Sigma,x)=1$. 
}

These results hold for general $\H^k$-rectifiable sets, but a particularity of the $1$-dimensional case is that we can use \Golab's Theorem to prove the convergence of blow-ups in the Hausdorff and Kuratowski topologies as well.
\begin{lemma}\label{lemma.blowup_domain_measure}
    Let $\Sigma \subset \mathbb{R}^d$ be closed and connected with $\H^1(\Sigma) < +\infty$, then for every $x\in \Sigma$ admitting an approximate tangent space $T_x\Sigma$ as in~\eqref{eq.blowup_measure}, and for all $R > 0$ it holds that
        \begin{equation}
            \displaystyle
            \frac{\Sigma - x}{r}\cap \overline{B_R(0)} \xrightarrow[r \to 0^+]{d_H} T_x\Sigma \cap \overline{B_R(0)},
        \end{equation}
        as well as global convergence holds in the Kuratowski sense 
        \[
            \frac{\Sigma - x}{r} \xrightarrow[r \to 0^+]{K} T_x\Sigma.
        \] 
        In addition, for every $r>$ it holds that
        \begin{equation}\label{eq.blowup_hausdorff_dist}
            d_H\left(
                \Sigma\cap B_r(x) -x, T_x\Sigma\cap B_r(0)
            \right)
            = 
            r 
            d_H\left(
                \frac{\Sigma - x}{r}\cap B_1,T_x\Sigma\cap B_1
            \right)
            =o(r).
        \end{equation}
\end{lemma}
\begin{proof}
    First we take a rectifiability point $x \in \Sigma$ with tangent space $T_x\Sigma$, which we know to be $\H^1$ a.a. of $\Sigma$, so that~\eqref{eq.blowup_measure} holds. Let $T$ be the (Kuratowski) limit of a subsequence $\displaystyle \frac{\Sigma- x}{r_k}$.
    From~\eqref{eq.blowup_measure} we have that $T_x\Sigma\subset T$.
    Thanks
    to Theorem~\ref{theorem.Golab_localversion}, for
    almost all $R>0$ it holds that
    \begin{equation}
        \H^1(T\cap B_R(0)) \le \liminf_{k\to\infty} \H^1\left(\frac{\Sigma-y}{r_k}\cap B_R(0)\right)
        = \H^1(T_y\Sigma\cap B_R(0)),\label{eq.blowup_domain_measure}
    \end{equation}
    which shows that $T\Delta T_x\Sigma$ is $\H^1$-negligible.
    
    Notice that, if there is some $z \in T\setminus T_x\Sigma$, we may consider some ball $B_s(z)$ which does not intersect $T_x\Sigma$. Since $T$ is the limit of connected sets, $z$ must be path-connected in $T$ to some point in $(B_s(z))^c$, so that $\H^1(T\cap B_s(z))\ge s$. This contradicts~\eqref{eq.blowup_domain_measure}. Hence, $T=T_x\Sigma$, and is independent of the subsequence, and we deduce the localized Hausdorff and the Kuratowski convergences.

    To check~\eqref{eq.blowup_hausdorff_dist}, notice that from homogeneity of the distance in $\mathbb{R}^d$ it holds that
    \[
        \frac{d_H\left(
                (\Sigma -x) \cap B_r, T_x\Sigma\cap B_r
            \right)}{r}
            = 
            d_H\left(
                \frac{\Sigma - x}{r}\cap B_1,T_x\Sigma\cap B_1
            \right)    
    \]
    and the RHS converges to zero as $r\to 0$ from the previous reasoning.
\end{proof}

\subsection{Loops and tree structure}\label{sec.preliminaries_loops}
We finally arrive at the central objects of the present work, which are loops from a connected set of finite length, or rather the absence of them. We start by properly defining what we mean by a loop.
\begin{definition}\label{def.loop}
    We say that a set $\Gamma$ is a \textit{loop} whenever it is homeomorphic to $\mathbb{S}^1$. Any connected set $\Sigma$ which contains no loops it is said to be a \textit{tree}. 

    A point $x \in \Sigma$ is a \textit{non-cut point of $\Sigma$} if $\Sigma \setminus \{x\}$ remains connected. Otherwise, $x$ is called a \textit{cut point}.
\end{definition}

It turns out that $\H^1$ almost every point in a loop is a non-cut point. This is proved for instance in~\cite[Lemma 5.6]{paolini2013existence} when the ambient space is a general metric space. In the following Lemma, we exploit the geometric structure of $\mathbb{R}^d$ to prove this result, while obtaining more information in the process. 
\begin{lemma}\label{lemma.noncut_property}
    Let $\Sigma \subset \mathbb{R}^d$ be a closed connected set with $\H^1(\Sigma) < +\infty$, consisting of more than one point and containing a loop $\Gamma$. Then $\H^1$-a.e.~point $x \in \Gamma$ is such that for any $r>0$ small enough, there exists 
    $
        \bar r \in \left(\frac{r}{2}, r\right),
    $
    such that 
    $\Sigma \setminus B_{\bar r}(x)$ and $\Sigma \cap B_{\bar r}(x)$ are connected and 
    \[
        \H^0(\Sigma \cap \partial B_{\bar r}(x)) 
        =
        \H^0(\Gamma \cap \partial B_{\bar r}(x)) = 2. 
    \]
    In addition, it holds that $\H^1$-a.e.~point of $\Gamma$ is a non-cut point. 
\end{lemma}
\begin{proof}
    Let $\Gamma$ be a loop of $\Sigma$, from the locality property of approximate tangent spaces~\cite[Proposition~10.5]{maggi2012sets}, we know that $\H^1$-a.e.~point of $\Gamma$ admits an approximate tangent plane such that 
    \[
        T_x\Sigma = T_x\Gamma. 
    \] 
    Fix one such point $x$ where the approximate tangents w.r.t.~$\Sigma$ and $\Gamma$ coincide and let $\mathbb{R}\tau$ be the common tangent space. Given $r>0$, it holds from the area formula and the blow-up Theorem that
    \begin{equation}\label{eq.bound_multiplicities}
        \int_0^{r} \H^0(\partial B_s(x) \cap \Gamma) \dd s
        \le 
        \int_0^{r} \H^0(\partial B_s(x) \cap \Sigma) \dd s
        \le 
        \H^1( B_r(x) \cap \Sigma)
        =
        2r + o(r).
    \end{equation}
    In addition, from the Hausdorff convergence of the blow-ups from $\Sigma \cap B_{r}(x)$, Lemma~\ref{lemma.blowup_domain_measure}, we can assume for $n$ large enough that
    \[
        \Sigma \cap B_{r}(x) \subset 
        \left\{
            z : 
            \begin{array}{c}
                \left|\inner{z - x, \tau} \right| < r\\
                \left|\inner{z - x, \tau^\perp} \right| < \frac{r}{100}
            \end{array}
        \right\}.
    \] 

    Since $\displaystyle \frac{\Gamma - x}{r}$ is a curve converging to the segment $\mathbb{R}\tau$, it must cross all the surfaces 
    \[
        \partial 
        \left(
            B_s(0) \cap \left\{\pm \inner{z, \tau} > 0\right\}
        \right)   
        \quad 
        0 < s < r,     
    \]
    so that $2 \le \H^0(\Gamma \cap \partial B_s(x)) \le \H^0(\Sigma \cap \partial B_s(x))$. As a result, from~\eqref{eq.bound_multiplicities} we have that
    \[
        0 \le \frac{1}{r} 
        \int_0^{r} 
        \underbrace{\left(\H^0(\partial B_s(x) \cap \Gamma) - 2\right)}_{\ge 0} \dd s
        \le 
        \frac{o(r)}{r}.
    \]
    Hence, for $r$ small enough, we can find 
    \[
        \bar r \in \left(\frac{r}{2}, r\right)
        \text{ such that }
        \H^0(\Sigma \cap \partial B_{\bar r}(x)) 
        =
        \H^0(\Gamma \cap \partial B_{\bar r}(x)) = 2. 
    \]
    For such radius we have that $\partial B_{\bar r}(x) \cap \Sigma = \partial B_{\bar r}(x) \cap \Gamma = \{y_{1,n}, y_{2,n}\}$ and $\Gamma \setminus B_{\bar r}(x)$ is a path between $y_{1,n}$ and $y_{2,n}$. 

    It follows that both $\Sigma \cap B_{\bar r}(x)$ and $\Sigma \setminus B_{\bar r}(x)$ remain connected. Indeed, for the former, it suffices to notice that since $\H^0(\Gamma \cap B_{\bar r}(x)) = 2$, $\Gamma \cap B_{\bar r}(x)$ is homeomorphic to an arc of $\mathbb{S}^1$ and so it is connected, as continuous images of connected sets are connected. As a result, it must also hold that $\Sigma \cap B_{\bar r}(x)$ is connected since if it was not, there would be a connected component $\Gamma'$ that is disjoint fom $\Gamma \cap B_{\bar r}(x)$. But since $\Sigma \cap \partial B_{\bar r}(x) = \Gamma \cap \partial B_{\bar r}(x)$, $\Gamma'$ would also be disjoint from $\Sigma \setminus B_{\bar r}(x)$, contradicting the connectedness of $\Sigma$. 
        
    To prove the connectedness of $\Sigma \setminus B_{\bar r}(x)$, consider $z_1, z_2 \in \Sigma \setminus B_{\bar r}(x)$ and let $\gamma \subset \Sigma$ be a path between them. If $\gamma \subset \Sigma \setminus B_{\bar r}(x)$, there is nothing to prove, otherwise $\gamma$ must contain either $y_{1,n}$ $y_{2,n}$, or both. If it contains only one of them, $\gamma\setminus B_{\bar r}(x)$ remains connected. In the case that it contains both, we can create a new path $\gamma \cup \Gamma \setminus B_{\bar r}(x)$ that must be connected, contained in $\Sigma \setminus B_{\bar r}(x)$ and has the points $z_1,z_2$. It follows that $\Sigma \setminus B_{\bar r}(x)$ is connected.

    Let us show that $x$ is a non-cut point. Indeed, for any $y_1, y_2 \in \Sigma \setminus \{x\}$, use the previous construction to obtain a radius such that $\Sigma \setminus B_{r}(x)$ is connected and contains $y_1, y_2$. Therefore, we can find a path in $\Sigma \setminus \{x\}$ connecting them proving that $\{x\}$ is a non-cut point.  
\end{proof}

As previously mentioned, Lemma~\ref{lemma.noncut_property} is a slight improvement over~\cite[Lemma 6.1]{buttazzo2003optimal} that is particularly useful to the localization arguments, since the latter provides a neighborhood $D_n$ around a.e.~non-cut point, but we have no information on the blowup of this set, complicating the implementation of the concentration/blow-up argument. With the construction provided by Lemma~\ref{lemma.noncut_property}, the limits of blow-up sequences are directly obtained via Lemma~\ref{lemma.blowup_domain_measure}.

\section{Absense of loops}\label{sec.absense_loops}
In this section we fix $\nu_\star \in \mathscr{P}(\mathbb{R}^d)$, a minimizer of problem~\eqref{problem.shape_optimization_relaxed}, along with its support $\Sigma$ and set $\alpha \eqdef \mathcal{L}(\nu_\star)$. We seek to  perform the construction that will show that  $\Sigma$ is a tree. We recall the two cases described in Section~\ref{subsection.localization_blowup_arg} for which this will be shown:
\begin{enumerate}
    \item[\underline{Case 1:}] if $\varrho_0$ is a convex combination of Dirac masses, \textit{i.e.}
    \[
        \varrho_0 = \mu_N = \sum_{i=1}^N a_i\delta_{x_i}, \text{ for }
        \sum_{i=1}^N a_i = 1. 
    \]
    \item[\underline{Case 2:}] $\varrho_0$ is absolutely continuous w.r.t.~the Lebesgue measure with compact support and bounded density, \textit{i.e.} $\varrho_0 \in L^\infty(\mathbb{R}^d)$.
    \item[\underline{Case 3:}]{  $\varrho_0$ is  of the form 
    \[
        \varrho_0 = \varrho + \sum_{i=1}^N a_i\delta_{x_i}, \text{ for $0 \le ai$ and }
        \sum_{i=1}^N a_i < 1,
    \]
    and $\varrho \in L^\infty(\mathbb{R}^d)$ with compact support.}
\end{enumerate}{ 
In both cases we can characterize the solutions to the relaxed problem~\eqref{problem.shape_optimization_relaxed} with the general result from~\cite[Theorem~6.4]{chambolle2025one}, which states that the solution is of the form
\begin{equation}\label{eq.chara.solution}
    \nu_\star = 
    \alpha^{-1}\mathscr{H}^1\mres \Sigma + 
    \varrho_{\text{exc}}\mres \Sigma,
\end{equation}
where $\alpha = \mathcal{L}(\nu_\star)$, and $\varrho_{\text{exc}}$ satisfies $\varrho_{\text{exc}} \le \varrho_0$ in the sense of measures. In Case 2, where $\varrho_0 \ll \mathscr{L}^d$, the second term above vanishes since $\Sigma$ has null Lebesgue measure. In Case 1, the solution might have atoms, but these are necessarily the same as $\varrho_0$, meaning a solution $\nu_\star$ is of the form
\begin{equation}\label{eq.characterization_atomic_source}
    \nu_\star 
    = 
    \alpha^{-1}\mathscr{H}^1\mres \Sigma
    + 
    \sum_{i = 1}^N b_i \delta_{x_i},
\end{equation}
where $0 \le b_i \le a_i$ for each $i = 1,\cdots,N$. 
}

In the course of the proof we will need to transport a part of the measure $\varrho_0$ with an arbitrary measurable selection of the projection 
operator
\begin{equation}\label{eq.projection_multimap}
    \Pi_\Sigma(x) = \argmin_{y \in \Sigma}\frac{1}{2}|x - y|^2.
\end{equation}  
Therefore, we assume that
\begin{equation}\label{eq.existence_measurable_selection}
    \text{ there is a measurable selection }
    \pi_\Sigma 
    \text{ of~\eqref{eq.projection_multimap} $\varrho_0$-a.e.~uniquely defined.}
\end{equation} 
This holds in 
\begin{itemize}
    \item Case 1, since for each $i$ we can choose $y_i \in \displaystyle \argmin_{\Sigma}|x_i - y|^2$ and define $\pi_\Sigma(x_i) \eqdef y_i$;
    \item Case 2, since the projection map is Lebesgue-a.e.~uniquely-defined.
    \item {  Case 3, by combining the two previous cases, namely assign values for the atoms and complete the rest with the Lebesgue-a.e.~uniquely-defined projection map. 
    }
\end{itemize}

\subsection{Loops are formed though projections}\label{subsec.loopsRprojections} 
In this paragraph we implement Step (1) of the concentration/blow-up argument described in Section~\ref{subsection.localization_blowup_arg} by showing that loops are formed through projections onto the optimal network, that is any optimal transportation plan $\gamma$ between $\varrho_0$ and a solution $\nu_\star$, whose support might contain a loop, satisfy the projection condition~\eqref{eq.formed_projections} below. 
\begin{proposition}\label{prop.projection_to_loops}
    Suppose that $\varrho_0$ has a compact support and that~\eqref{eq.existence_measurable_selection} holds. Let $\nu_\star$ be a minimizer of~\eqref{problem.shape_optimization_relaxed}. If $\gamma$ is an optimal transportation plan between $\varrho_0$ and $\nu_\star$ and $\Gamma \subset \Sigma$ is a loop, we then have that
    \begin{equation}\label{eq.formed_projections}
        |x-y| = \dist(x,\Sigma) \text{ for $\gamma$-a.e.~$(x,y) \in {\mathbb{R}}^d\times \Gamma$}. 
    \end{equation}
\end{proposition}
\begin{proof}
    {  
    Following the analysis of~\cite[Lemma~5.1 and Thm~6.4]{chambolle2025one}, we can show that we can decompose the original measure $\varrho_0$ and any solution $\nu_\star$ as 
    \[
       \varrho_0 = \varrho_{\mathscr{H}^1} + \varrho_{\text{exc}}, \
       \nu_\star = \nu_{\mathscr{H}^1} + \nu_{\text{exc}} 
    \]
    which have the properties that $\varrho_{\mathscr{H}^1} \eqdef \alpha^{-1}\mathscr{H}^1\mres \Sigma$, $\supp \varrho_{\text{exc}} \subset \Sigma$, and $\nu_{\text{exc}} = \varrho_{\text{exc}} \mres \Sigma = \varrho_{\text{exc}}$. 
    The intuition on this notation is that $\varrho_{\text{exc}}$ is the excess of the measure that is uniformly distributed. In Case 2, this excess is irrelevant, but it is important in Case 1, as seen above, and must be taken into account in order to prove this Proposition in full generality. 

    This decomposition is induced by a decomposition of any optimal transportation plan $\gamma \in \Pi(\varrho_0, \nu_\star)$
    \[
        \gamma = \gamma_{\mathscr{H}^1} + \gamma_{\text{exc}},
    \]  
    with the property that $\gamma_{\mathscr{H}^1} \in \Pi(\varrho_{\mathscr{H}^1}, \nu_{\mathscr{H}^1})$ and $\gamma_{\text{exc}} \in \Pi(\varrho_{\text{exc}}, \nu_{\text{exc}})$ are both optimal transportation plans between their marginals, w.r.t.~the cost $|x-y|^p$. 

    As a result, $\gamma_{\text{exc}}$ a.e. $(x,y)$ trivially satisfy the desired projection property~\eqref{eq.formed_projections}. To deal with the contribution of $\gamma_{\mathscr{H}^1}$, we first define the following set
    \begin{equation}
        \bar \Gamma 
        \eqdef 
        \left\{
            x \in \Gamma : \theta_1(\nu_\star,x) = \alpha^{-1}
        \right\},
        \text{  where $\alpha = \mathcal{L}(\nu_\star)$.}
    \end{equation}
    From the characterization of solutions~\eqref{eq.chara.solution}, $\gamma_{\mathscr{H}^1} \mres \mathbb{R}^d \times \Gamma$ is actually concentrated in $\mathbb{R}^d \times \bar \Gamma$, so we can focus on this set from now on. 

    The strategy is to show that any transport pair $(x,y)$ in the “bad set” defined for $\delta > 0$ as 
    \[
        E_\delta 
        \eqdef 
        \left\{
            (x,y) \in \mathbb{R}^d \times \bar \Gamma : 
            \ |x-y|^p > {\dist(x,\Sigma)}^p + \delta            
        \right\}, 
    \]
    has zero contribution for all $\delta>0$. The ``good set'' is then given by
    \[
        F \eqdef 
        \left\{
            (x,y) \in \mathbb{R}^d \times \bar \Gamma : 
            \ |x-y|^p = {\dist(x,\Sigma)}^p
        \right\}. 
    \]

    Setting the measures $\nu_\delta, \nu_F$ defined for a Borel set $A$ as 
    \begin{equation}\label{eq.nu_delta}
        \nu_\delta(A) \eqdef \gamma \mres E_\delta (\mathbb{R}^d \times A),
        \quad 
        \nu_F(A) \eqdef \gamma\mres F (\mathbb{R}^d \times A),
    \end{equation}
    it follows by construction that $\nu_\delta, \nu_F \le \nu_\star$. Indeed, for any $0 \le f \in \mathscr{C}_b(\mathbb{R}^d)$ we have that
    \[
        \int_{\mathbb{R}^d} f \dd \nu_\delta
        = 
        \int_{E_\delta} f \dd \gamma
        \le 
        \int_{\mathbb{R}^d \times \mathbb{R}^d} f \dd \gamma
        = 
        \int_{\mathbb{R}^d} f \dd \nu_\star. 
    \]
    Similarly for $\nu_F$. Given the characterization of solutions given by~\eqref{eq.chara.solution} and the fact that we have removed the possible singular parts of $\nu_\star$ w.r.t.~$\mathscr{H}^1$ by only considering the mass over $\bar \Gamma$, this implies that $\nu_\delta, \nu_F$ are $1$-rectifiable measures. In addition, a similar argument using the monotone convergence theorem shows that 
    \begin{equation}
        \nu_\star \mres \bar \Gamma 
        = 
        \nu_F + 
        \sup_{\delta > 0} \nu_\delta. 
    \end{equation}
    Therefore, the result will follow if we show that $\nu_\delta \equiv 0$, for all $\delta > 0$. Indeed, let ${\left(\gamma^y\right)}_{y \in \Sigma}$ denote a disintegration of $\gamma$ w.r.t.~its second marginal $\nu_\star$, we have
    \[
        \gamma(\mathbb{R}^d\times \bar \Gamma) 
        = 
        \int_{\bar \Gamma} \gamma^y(\mathbb{R}^d)\dd \nu_\star
        = 
        \int_{\bar \Gamma} \gamma^y(\mathbb{R}^d)\dd \nu_F 
        = 
        \gamma(F). 
    \]
    Since $F \subset \mathbb{R}^d\times \bar \Gamma$, it must hold that $\gamma(\mathbb{R}^d\times \bar \Gamma \setminus F) = 0$ and the result follows. To conclude it suffices to show that for any $\bar y \in \bar\Gamma$, admitting an approximate tangent space $T_{\bar y}\Sigma = T_{\bar y} \Gamma$ it holds that
    \[
        \theta_1(\nu_\delta, \bar y) = 0.
    \]
    As $\nu_\delta$ is a rectifiable measure, this will imply that $\nu_\delta \equiv 0$.}

    Fix $\bar y \in \bar \Gamma$ a point as above, which exists thanks to Lemma~\ref{lemma.noncut_property}. Let ${\left(r_n\right)}_{n \in \mathbb{N}}$ be an infinitesimal sequence obtained from Lemma~\ref{lemma.noncut_property} such that $\Sigma_n \eqdef \Sigma \setminus B_{r_n}(\bar y)$ remains connected. For $n$ large enough, let us show that if 
    \begin{equation}\label{eq.projection_far_from_ball}
        (x,y) \in E_\delta\cap (\mathbb{R}^d \times B_{r_n}(\bar y))
        \text{ then }
        \pi_\Sigma(x) \in {  \Sigma_n}.
    \end{equation}
    Indeed, for such a pair $(x,y)$ we have that
    \begin{align*}
        {\dist(x,\Sigma)}^p + \delta
        &\le 
        |x - y|^p 
        \le 
        {\left(
            \dist(x,\Sigma) + |\pi_\Sigma(x) - y|
        \right)}^p \\ 
        &\le 
        {\dist(x,\Sigma)}^p + p{\left(\dist(x,\Sigma) + |y - \pi_\Sigma(x)|\right)}^{p-1}|y - \pi_\Sigma(x)|\\
        &\le 
        {\dist(x,\Sigma)}^p + { C_{p,\varrho_0,\Sigma}}|y - \pi_\Sigma(x)|,
    \end{align*} 
    where the third inequality follows from the convexity of $t\mapsto |t|^p$, and the constant is defined as {  $C_{p,\varrho_0,\Sigma} \eqdef p(2\dist(\supp\varrho_0,\Sigma))^{p-1}$, which is finite since both $\supp\varrho_0$ and $\Sigma$ are compact}. As a result, for $n$ sufficiently large, we obtain that
    \[
        2 r_n 
        < \frac{\delta}{{ C_{p,\varrho_0,\Sigma}}} 
        \le |y - \pi_\Sigma(x)|.    
    \]
    Since $y \in B_{r_n}(\bar y)$, it must follow that $\pi_\Sigma(x) \in \Sigma \setminus B_{r_n}(\bar y)$, for $n$ large enough. 

    In the sequel, we write $B_{r_n} = B_{r_n}(\bar y)$ to simplify notation, and we define an alternative transportation plan as follows
    \begin{equation}\label{eq.competitor_transport_plan}
        \gamma' 
        \eqdef
        \gamma \mres \mathbb{R}^d \times \Sigma_n
        + 
        {(\pi_0 , \pi_\Sigma \circ \pi_0)}_\sharp \gamma 
        \mres E_\delta \cap (\mathbb{R}^d \times B_{r_n})
        + 
        {(\pi_0 , y_n)}_\sharp\gamma \mres (\mathbb{R}^d \times B_{r_n}) \setminus E_\delta,
    \end{equation}
    where $\pi_0$ denote the projection onto the first marginal, \textit{i.e} $\pi_0(x,y) = x$, and $y_n$ { is an arbitrary point of} $\Sigma_n \cap \partial B_{r_n}(\bar y)$. Its second marginal then defines a new competitor as {  $\nu' \eqdef {(\pi_1)}_\sharp\gamma'$, which can be characterized via duality as 
    \begin{align*}
        \int_{\mathbb{R}^d} f \dd \nu'
        &=
        \int_{\mathbb{R}^d \times \Sigma_n} f(y) \dd \gamma
        +
        \int_{E_\delta \cap (\mathbb{R}^d \times B_{r_n})} f(\pi_\Sigma(\pi_0(x,y))) \dd \gamma
        +
        \int_{(\mathbb{R}^d \times B_{r_n}) \setminus E_\delta} f(y_n) \dd \gamma\\ 
        &= 
        \int_{\mathbb{R}^d\times \mathbb{R}^d} f(y) \dd \gamma'. 
    \end{align*}
    As a result we get
    \begin{equation}\label{eq.competitor}
        \nu' 
        \eqdef  
        \nu_\star \mres \Sigma_n + 
        \nu_\delta \mres B_{r_n}
        +
        \gamma \left(
            (\mathbb{R}^d \times B_{r_n}) \setminus E_\delta
        \right)\delta_{y_n}.
    \end{equation}    
    }

    {  In the transportation plan defined in~\eqref{eq.competitor_transport_plan}, the first term preserves the transportation plan that forms $\Sigma_n$; thanks to the definition of $E_\delta$, the second projects onto $\Sigma$ the mass that is sent to $\Sigma \cap B_{r_n}$ with a transportation that exceeds the distance to $\Sigma$ in at least $\delta$, and the last term sends the remaining mass whose projection is close to $\Sigma\cap B_{r_n}$ to the point $y_n$, creating a Dirac mass at $y_n$.} 
    
    Since the mass on the second term of the transportation plan $\gamma'$ in~\eqref{eq.competitor_transport_plan} is sent to $\Sigma_n$, it follows that $\supp \nu' = \Sigma_n$. But since this operation can only increase the density of $\nu_\star$ over $\Sigma_n$, we have that $\nu' \mres \Sigma_n \ge \nu_\star \mres \Sigma_n$ and it follows that
    \begin{equation}\label{eq.length_functional_inequality}
        \mathcal{L}(\nu_\star) \ge \mathcal{L}(\nu').
    \end{equation} 
    
    This construction yields
    \begin{align*} 
        W_p^p(\varrho_0, \nu_\star)
        &= 
        \int_{\mathbb{R}^d\times \Sigma_n}|x-y|^p\dd \gamma 
        + 
        \int_{\mathbb{R}^d\times B_{r_n} \cap E_\delta} |x-y|^p\dd \gamma+
        \int_{\mathbb{R}^d\times B_{r_n} \setminus E_\delta} |x-y|^p\dd \gamma\\ 
        &\ge
        \int_{\mathbb{R}^d\times \Sigma_n}|x-y|^p\dd \gamma 
        + 
        \int_{\mathbb{R}^d\times B_{r_n} \cap E_\delta} 
        \left({\dist(x,\Sigma)}^p + \delta\right)\dd \gamma\\ 
        &+ 
        \int_{\mathbb{R}^d\times B_{r_n} \setminus E_\delta} |x - y_n|^p\dd \gamma 
        - 
        p \int_{\mathbb{R}^d\times B_{r_n} \setminus E_\delta} 
        \underbrace{\left| |x - y_n| - |x - y| \right|}_{\le |y - y_n| \le 2r_n} |x - y_n|^{p-1} \dd \gamma\\ 
        \ge&
        \int_{\mathbb{R}^d\times \mathbb{R}^d} |x - y|^p\dd \gamma' 
        + 
        \delta  \nu_\delta(B_{r_n}) 
        - 2p r_n 
        \int_{\mathbb{R}^d\times B_{r_n} \setminus E_\delta}
        |x - y_n|^{p-1}\dd \gamma,
    \end{align*}
    { 
    where from the first to second line we have used the convexity of $t \mapsto t^p$ for $p \ge 1$. From the minimality of $\nu_\star$ and~\eqref{eq.length_functional_inequality} we have that
    \[
        W_p^p(\varrho_0, \nu_\star) 
        \le 
        W_p^p(\varrho_0, \nu')
        \le 
        \int_{\mathbb{R}^d\times \mathbb{R}^d} |x - y|^p\dd \gamma'.
    \]
    As a result, the previous estimate gives a bound on the $1$-density of the measure $\nu_\delta$, defined in~\eqref{eq.nu_delta}, which we were interested in: 
    \begin{align*}
        \frac{ \nu_\delta(B_{r_n}(\bar y)) }{2r_n}
        &\le 
        \frac{p}{\delta}
        \int_{\mathbb{R}^d\times B_{r_n} \setminus E_\delta}
        |x - y_n|^{p-1}\dd \gamma
        \cvstrong{n \to \infty}{} 0,
    \end{align*}
    where the last limit follows from the fact that the integrand is bounded and 
    \[
        \gamma(\mathbb{R}^d\times B_{r_n}(\bar y)) 
        = \nu_\star(B_{r_n}(\bar y)) 
        = \theta_1(\nu_\delta, \bar y) r_n + o(r_n) 
        = \alpha^{-1} r_n + o(r_n)
        \cvstrong{n \to \infty}{} 0. 
    \]
    We conclude that for $\mathscr{H}^1$-a.e.~point $\bar y$ of $\bar \Gamma$, it holds that
    $
        \theta_1(\nu_\delta, \bar y) = 0,
    $
    and the result follows. }
\end{proof}
{
      
    \begin{remark}
        Notice that the above proposition is very general, it covers any original measure $\varrho_0$ with compact support, and for which we can construct a Borel measurable selection of the projection multimap $\Pi_\Sigma$. In pathological cases where $\varrho_0$ concentrates on sets of dimension between $0$ and $1$, the thesis is still valid, but it is not clear how to exploit this result to prove absence of loops. Indeed, the tree property is not true in general if $\varrho_0$ is a uniform measure on a circle since then the solution $\nu_\star$ has incentive to concentrate on this circle, as indicates the characterization~\eqref{eq.chara.solution}. 

        The difficulty is to deal with the set $\bar \Gamma$, since it might even be dense in $\Gamma$, in which case the construction via the concentration/blow-up argument following this Proposition is not feasible. This set coincides with $\Gamma$, up to a set of zero $\mathscr{H}^1$ measure, under the same conditions that are sufficient to have existence of solutions to the original problem~\eqref{problem.shape_optimization}, when $\varrho_0$ does not give mass to sets of dimensional smaller or equal to $1$. In the case that $\varrho_0$ has atoms at ${\left(x_i\right)}_{i = 1}^N$, or for instance composed of an atomic plus absolutely continuous part, we have that $\bar \Gamma = \Gamma \setminus {\left(x_i\right)}_{i = 1}^N$.
    \end{remark}
}

\subsection{Localizations and blow-up}\label{subsec.blow-up}
Since we know from Prop~\ref{prop.projection_to_loops} that loops are formed though projections, we can perform Step (2) from Section~\ref{subsection.localization_blowup_arg}. That is, we chose a suitable point to perform localizations. 

As the proof is by contradiction, we first assume that $\Sigma$ contains a loop $\Gamma$. We consider 
\begin{equation}\label{eq.assumptions_y_0}
    y_0 \in \Gamma, \text{ is a noncut point such that } T_{y_0}\Sigma = T_{y_0}\Gamma,
\end{equation}
which can be done since, $\mathscr{H}^1$-a.e., the approximate tangent spaces to $\Sigma$ and $\Gamma$ coincide. In Cases 1 and 3, where $\varrho_0$ has atoms, we make the additional assumption
\begin{equation}
    y_0 \neq x_i, \text{ for all } i = 1,\dots,N.
\end{equation}
Next, let ${\left(r_n\right)}_{n \in \mathbb{N}}$ be a sequence of radii obtained from Lemma~\ref{lemma.noncut_property}, and we introduce the following notation
\begin{equation}
    \Sigmayrn \eqdef \Sigma \cap B_{r_n}(y_0), \ 
    \Sigma_n \eqdef \Sigma \setminus \Sigmayrn,
\end{equation}
so that from Lemma~\ref{lemma.noncut_property} it holds that
\begin{equation}
    \Sigmayrn \text{ and }\Sigma_n \text{ are connected and } r_n \to 0. 
\end{equation}

In the sequel, we will focus our attention to the following sequence of localized measures 
\[
    \nu_n \eqdef \nu_\star \mres \Sigmayrn.    
\]
From the optimality of $\nu_\star$, this sequence minimizes a family of localized variational problems consisting of the transportation of ``the portion of $\varrho_0$ that is sent to $\nu_n$'', namely 
\[
    \varrho_n \eqdef 
    {(\pi_0)}_\sharp 
    \left(
        \gamma\mres \mathbb{R}^d\times\Sigmayrn
    \right).  
\]
In Case 2, we can equivalently write $\varrho_n = \varrho_0\mres {T}^{-1}(\Sigmayrn)$, where $T$ corresponds to the optimal transportation map from $\varrho_0$ to $\nu_\star$. 

Afterward, we define a blow-up of this sequence of problems and extract a limit. But to prevent the measure $\varrho_n$ from losing mass at infinity in the blow-up step, as in~\cite{chambolle2025one} we perform the concentration operation from Step (3) of Section~\ref{subsection.localization_blowup_arg}. More specifically, we let $\varrho_n$ evolve along a constant speed geodesic in the Wasserstein space almost until it reaches $\nu_n$, defined as follows: if $\gamma_n$ is an optimal transportation plan between $\varrho_n$ and $\nu_n$, we are interested in the following geodesic interpolation between them
\begin{equation}\label{eq.sigma_n}
    \begin{aligned}
        \sigma_n 
        &\eqdef 
        {\left(\pi_{r_n}\right)}_\sharp\gamma_n 
        \text{ where }
        \pi_{r_n}\eqdef 
        r_n \pi_0 + (1-r_n)\pi_1.
    \end{aligned}
\end{equation}
The reader is referred to~\cite[Theorem~5.27]{santambrogio2015optimal} for a proof of the fact that the above interpolation indeed yields geodesics for the $W_p$ distance. 

With these elements we obtain the following result, whose proof is included in Appendix~\ref{appendix} for completeness since it is a minor variant of the results found in~\cite{chambolle2025one}. But as we are interested in making variations that will ``open'' the loop $\Gamma$, to simplify the notation we define the following class of sets
\begin{equation}\label{eq.variations_2connected_components}
    \mathcal{A}_2 
    \eqdef 
    \left\{
        \Sigma' \subset \overline{B_1(0)} : 
            \text{ $\Sigma'$ has at most 2 connected components}
    \right\}.
\end{equation}
\begin{lemma}\label{lemma.localization}
    The localized measure $\nu_n$ solves the following minimization problem,
    \begin{equation}\label{eq.localized_problem}
        \min \left\{
            W_p^p(\sigma_n, \nu'): 
            \substack{
                \displaystyle
                \text{ there is }
                \Sigma' \in \mathcal{A}_2 
                \text{ such that }\\  
                \\
                \displaystyle
                \nu' \in \mathscr{M}_+(\Sigma'), \ 
                \nu' \ge \alpha^{-1}\H^1\mres \Sigma',\\ 
                \\
                \displaystyle
                \Sigma_n \cup \Sigma' \text{ is connected,}\\ 
                \\
                \displaystyle
                \ \nu'(\overline{B_{1}(0)}) = 
                \nu_\star\left(\Sigmayrn\right)\\
            }           
        \right\},
    \end{equation}
    {  where we recall the notation $\alpha \eqdef \mathcal{L}(\nu_\star)$}
\end{lemma}

In the sequel, recalling the definition of the blow-up operator $\Phi^{y_0,r} = \frac{\id - y_0}{r}$ from~\eqref{eq.blowup_measure} in Section~\ref{sec.golab_length_tangent}, we notice that for any given measures $\mu, \nu$ it holds that
\begin{equation}\label{blowup_Wasserstein_identity_qp}
    W_p^p
    \left(
        \frac{1}{r}{(\Phi^{y_0, r})}_\sharp \mu, 
        \frac{1}{r}{(\Phi^{y_0, r})}_\sharp \nu
    \right) 
    = 
    \frac{1}{r^{p+1}} 
    W_p^p\left( \mu,  \nu\right). 
\end{equation}
We are particularly interested in the sequences of blow-ups of the measures $\sigma_n$ and $\nu_n$:
\begin{equation}\label{blowup_measures_renormalized_qp}
    \bar \sigma_n \eqdef \frac{1}{r_n}{(\Phi^{y_0,r_n})}_\sharp \sigma_n, 
    \quad
    \bar \nu_n \eqdef \frac{1}{r_n}{(\Phi^{y_0,r_n})}_\sharp \nu_n,
\end{equation}
since we already know from Lemma~\ref{lemma.localization} that they will inherit some optimality property. 

From Lemma~\ref{lemma.localization} and~\eqref{blowup_Wasserstein_identity_qp}, each element from the sequence ${\left(\bar \nu_n\right)}_{n \in \mathbb{N}}$ is a minimizer of a sequence of functionals ${\left(F_n\right)}_{n \in \mathbb{N}}$ defined as
\begin{equation}\label{eq.family_Fn}
    F_n(\nu') \eqdef
    \begin{cases}
        W_p^p
        \left(  
        \bar \sigma_{n},  
        \nu' 
        \right),&
        \substack{
                \displaystyle
                \text{ there is }
                \Sigma' \in \mathcal{A}_2 
                \text{ such that }\\  
                \\
                \displaystyle
                \nu' \in \mathscr{M}_+(\Sigma'), \ 
                \nu' \ge \alpha^{-1}\H^1\mres \Sigma',\\ 
                \\
                \displaystyle
                \left(\frac{\Sigma_n - y_0}{r_n}\right) \cup \Sigma' \text{ is connected,}\\ 
                \\
                \displaystyle
                \ \nu'(\overline{B_{1}(0)}) = 
                \frac{\nu_\star\left(\Sigmayrn\right)}{r_n},\\
            }
        \\ 
        +\infty,& \text{ otherwise.}
    \end{cases}
\end{equation}
Now, recall that from the blow-up properties of $\Sigma$, if follows that
\[
    \frac{\Sigma_{y_0,r_n} - y_0}{r_n} \cvstrong{n \to \infty}{d_H} T_{y_0}\Sigma \cap \overline{B_1(0)}. 
\]
We can also extract a subsequence for the convergence of the measures, so that it holds that 
\begin{equation}\label{eq.limit_blowup_measures}
    \bar \sigma_n \cvstar{n \to \infty} \bar \sigma, \quad 
    \bar \nu_n \cvstar{n \to \infty} \bar \nu,
\end{equation}
in $\overline B_1(0)$. This motivates the following limit problem, which is minimized by $\bar \nu$ as we shall prove later, 
\begin{equation}\label{eq.limit_F}
    F(\nu') \eqdef
    \begin{cases}
        W_p^p
        \left(  
        \bar \sigma,  
        \nu'
        \right),&
        \substack{
                \displaystyle
                \text{ there exists }
                \Sigma' \in \mathcal{A}_2
                \text{ such that }\\
                \\
                \displaystyle
                \nu' \in \mathscr{M}_+(\Sigma'), \ 
                \nu' \ge \alpha^{-1}\H^1\mres \Sigma',\\ 
                \\
                \displaystyle
                T_{y_0}\Sigma \cap \partial B_1(0) \subset \Sigma',\\ 
                \\
                \displaystyle
                \ \nu'(\overline{B_{1}(0)}) = 2\theta_1(\nu_\star, y_0),\\
            }
        \\ 
        +\infty,& \text{ otherwise.}
    \end{cases}
\end{equation}
{  Before proceeding let us make a comment on the constraints on $F$. The condition $T_{y_0}\Sigma \cap \partial B_1(0) \subset \Sigma'$ comes from the connectedness constraint from $F_n$, so one would wonder if the connectedness constraint on $F_n$ would not pass to the limit as a connectedness constraint as well. We have crafted our class of variations in order to ``open the loop'' in a small ball, allow for variations inside it, and then perform the blow-up. Since the variations are contained in a ball $B_{r_n}$ and $r_n$ is the scale of the blow-up, this gives rise to the sets $\Sigma'$ in the definition of $F$. But the blowup outside of $B_{r_n}$ flattens $\Sigma$ in such a way that $\Sigma\setminus B_{r_n}$ becomes two disjoint segments. As a result, imposing for instance for $\Sigma' \cup (T_{y_0}\Sigma\setminus B_1(0))$ to be connected would mean closing back the loop.}

Step (3) described in Section~\ref{subsection.localization_blowup_arg} consists of defining the functionals $F_n$ above and show that they $\Gamma$-converge to $F$. This is done in the following Theorem, whose proof is also left to the Appendix~\ref{appendix}. 
\begin{theorem}\label{theorem.gamma_conv_Fn}
    The family ${\left(F_n\right)}_{n \in \mathbb{N}}$ converges to $F$ in the sense of $\Gamma$-convergence, for the topology of weak-$\star$ convergence of Radon measures. 
\end{theorem}

In Step (4), we transfer a lot of information about the minimization of $F_n$ to the minimization of $F$, by means of the $\Gamma$-convergence result and the fact that the optimal transportation in the definition of $F_n$ is almost achieved via projections. In fact, only the transportation onto $\Gamma \cap B_{r_n}(y_0)$ is given by projections, and there might be some mass in the set $(\Sigma\setminus \Gamma)\cap B_{r_n}(y_0)$. But since $\Sigma$ and $\Gamma$ have the same approximate tangent space at $y_0$, this contribution vanishes as $n\to \infty$, and the limit 
inherits the projection properties from the loop $\Gamma$. 

{ 
    In the following Lemma we formalize the previous discussion, and to simplify notation we assume without loss of generality that the tangent space of $\Sigma$ at $y_0$ corresponds to the vertical axis in $\mathbb{R}^d$, that is 
    \begin{equation}
        T_{y_0}\Sigma = \mathbb{R}e_d, 
    \end{equation}
    where ${(e_i)}_{i=1}^d$ denotes the canonical euclidean basis of $\mathbb{R}^d$. The previous discussion is formalized in the following Lemma.
}
\begin{lemma}\label{lemma.properties_Gamma_limit}
    The following assertions are true:
    \begin{enumerate}
        \item[(i)] We have $\bar \nu = 2\theta_1(\nu_\star,y_0)\H^1\mres [-e_d, e_d]$ and it is a minimizer of $F$;
        { 
        \item[(ii)] In Case 1, when $\varrho_0$ is atomic, define the quantity 
            \begin{equation*}
                0 < L \eqdef \min_{i = 1,\dots,N} |y_0 - x_i|.
            \end{equation*}
        Then we have that $\supp \bar \sigma \subset \{ \dist(\cdot, T_{y_0}\Sigma) \ge L\}$.
        \item[(iii)] In Cases 2 and 3, when $\varrho_0$ is absolutely continuous with a density $\varrho_0 \in L^\infty(\Omega)$, or the mixture of an atomic measure and a bounded density with compact support, consider the following cylindrical sets 
        \begin{equation}\label{eq.cylinders}
            C_{\delta,\varepsilon} 
            \eqdef 
            \left\{
                (x',x_d) \in \mathbb{R}^d\times \mathbb{R} : 
                \substack{
                    \displaystyle
                    \dist(x',\mathbb{R}e_d) < \delta, \\ 
                    \displaystyle
                    |x_d| < \varepsilon
                }
            \right\},
        \end{equation}
        then we have the following estimate for $\bar \sigma$
        \[
           \bar \sigma(C_{\delta,\varepsilon}) \le  2\omega_{d-1}\norm{\varrho_0}_\infty\varepsilon\delta^{d-1}, 
        \]
        in Case 2 and the respective bound works for $\delta, \varepsilon$ small enough in Case 3 with $\varrho_0$ replaced with its bounded density component.
        In particular, $\bar \sigma(T_{y_0}\Sigma \cap \overline{B_1(0)}) = 0$
        }
        \item[(iv)] The optimal transportation from $\bar \sigma$ to $\bar \nu$ is attained by the projection map onto $T_{y_0}\Sigma$. 
    \end{enumerate}
\end{lemma}
\begin{proof}
    Starting with item $(i)$, recall that 
    \[
        \bar \nu_n = 
        \frac{1}{r_n}{\left(\Phi^{y_0,r_n}\right)}_\sharp\nu_n,
    \] 
    where $\nu_n$ is a minimizer of~\eqref{eq.localized_problem} thanks to Lemma~\ref{lemma.localization}. As a result, $\Sigmayrn$ satisfies the restrictions of~\eqref{eq.localized_problem}. As a result, the set 
    $\displaystyle
        \frac{\Sigmayrn - y_0}{r_n}
    $
    satisfy all the restrictions of $F_n$ for $\bar \nu_n$. On the other hand, given any $\varrho$ satisfying the restrictions of $F_n$ with a set $\Sigma'$ yields $\nu' = r_n{\left(\Phi^{y_0,r_n}\right)}^{-1}_\sharp \varrho$ admissible for~\eqref{eq.localized_problem} with the set $y_0 + r_n \Sigma'$. Indeed, the only property that requires checking is that $\nu' \ge {\alpha}^{-1}\mathscr{H}^1\mres (y_0 + r_n \Sigma')$. This follows from a simple change of variables since, for any continuous $\phi \ge 0$, we have
    \begin{align*}
        \int \phi \dd \nu' 
        &= 
        r_n \int \phi(y_0 + r_n x)\dd \varrho(x) 
        \ge 
        {\alpha}^{-1}r_n \int_{\Sigma'} \phi(y_0 + r_n x) \dd \H^1(x) \\ 
        &= 
        {\alpha}^{-1}r_n \int_{y_0 + r_n\Sigma'} \phi\dd \H^1.
    \end{align*}

    As a result, using identity~\eqref{blowup_Wasserstein_identity_qp}, it follows that 
    \begin{align*}
        W_p^p(\bar \sigma_n, \bar \nu_n) 
        &= 
        \frac{1}{r_n^{p+1}}
        W_p^p(\sigma_n, \nu_n) 
        \le 
        \frac{1}{r_n^{p+1}}
        W_p^p(\sigma_n, \nu)\\ 
        &\le
        W_p^p(\bar \sigma_n, \varrho),  
    \end{align*}
    showing that $\bar \nu_n$ is a sequence of minimizers, so that the minimality of $\bar \nu$ follows from the fundamental properties of $\Gamma$ convergence.
    {  To finish the proof of item (i), we recall that the sequence $\bar \nu_n$ is of the form 
    \[
        \bar \nu_n = \frac{1}{r_n}{\Phi^{y_0,r_n}}_\sharp(\nu_\star\mres \Sigmayrn),
    \]
    and that $\nu_\star$ is uniformly distributed over $\Sigmayrn$ thanks to the characterization of solutions~\eqref{eq.chara.solution} and the choice of $y_0$. As a result, from the blow-up Theorem~\eqref{eq.blowup_measure} and the definition of $\theta_1(\nu_\star,y_0)$ we conclude. 
    }

    {  Item $(ii)$ follows directly from the fact that $\varrho_0$ is atomic. To prove item $(iii)$,} first we recall that since $\varrho_0$ is absolutely continuous, its optimal transportation is uniquely attained by a map $T$, and we can write $\sigma_n = {T_{r_n}}_\sharp \varrho_n$, with $T_{r_n} \eqdef r_n \id + (1-r_n)T$ and $\varrho_n$-a.e.~$T = \pi_\Sigma$, thanks to Proposition~\ref{prop.projection_to_loops}. 
    { 
    To simplify our analysis in the sequel, using Lemma~\ref{lemma.noncut_property} we can find a sequence ${\left(\varepsilon_n\right)}_{n \in \mathbb{N}}$ such that 
    \[
        \frac{\varepsilon_n}{r_n} \searrow \varepsilon,
    \]
    and having the property that both $\Sigma_{y_0,\varepsilon_n}$ and $\Sigma\setminus \Sigma_{y_0,\varepsilon_n}$ are connected and 
    \begin{equation}\label{eq.multiplicity_2_lemma_proj}
        \Sigma_{y_0,\varepsilon_n} \cap \partial B_{\varepsilon_n}(y_0)
        = 
        \{y_{n,1}, y_{n,2}\}
        =
        \Sigma\setminus \Sigma_{y_0,\varepsilon_n} \cap \partial B_{\varepsilon_n}(y_0).
    \end{equation}
    
    Since the cylindrical sets $C_{\delta,\varepsilon}$ defined in~\eqref{eq.cylinders} are all open, we have for all $\delta,\varepsilon > 0$ that 
    \begin{equation}\label{eq.estimate_Cdelta}
        \bar \sigma(C_{\delta,\varepsilon}) 
        \le 
        \liminf_{n \to \infty}
        \bar \sigma_n(C_{\delta,\varepsilon})
        \le 
        \liminf_{n \to \infty}
        \bar \sigma_n\left(
            C_{\delta,\frac{\varepsilon_n}{r_n}}
        \right),
    \end{equation}}
    where by definition we have that
    \[
        \bar \sigma_n\left(
            C_{\delta,\frac{\varepsilon_n}{r_n}}
        \right)
        = 
        r_n^{-1}
       \varrho_n
        \left(
            {T_{r_n}}^{-1}(y_0 + C_{\delta r_n,\varepsilon_n})
        \right).
    \]

    Hence, let us study the set ${T_{r_n}}^{-1}(y_0 + C_{\delta r_n,\varepsilon_n})$. Consider a pair $(x,y)$ such that $y \in y_0 +  C_{\delta r_n,\varepsilon_n}$, $x \in \supp \varrho_n$ and 
    \begin{equation}\label{eq.def_y}
        y = T_{r_n}(x) = r_n x + (1-r_n)T(x). 
    \end{equation}
    Since for $\varrho_n$-a.e.~$x$, the map $T$ behaves as a projection onto $\Sigma$, and the map $T_{r_n}$ is an interpolation between the identity and the projection onto $\Sigma$, it follows that $T(x) = \pi_\Sigma(x) = \pi_\Sigma(y)$. {  Thanks to property~\eqref{eq.multiplicity_2_lemma_proj}, necessarily the projections satisfy $\pi_\Sigma(x) \in B_{\varepsilon_n}(y_0)$ for $\varrho_n$-a.e. $x$.}
    
    In addition, rearranging the terms in~\eqref{eq.def_y} we obtain
    \begin{align*}
        r_n(x - T(x)) 
        = 
        y - T(x)
        = 
        r_n
        \left(
            \frac{y - y_0}{r_n} - \frac{\pi_\Sigma(y) - y_0}{r_n}
        \right)
    \end{align*}
    so that recalling that $y \in y_0 + r_n C_{\delta,\varepsilon}$, it holds that 
    \begin{align*}
        \dist(x, \Sigma_{y_0,\varepsilon_n}) 
        &= 
        |x - T(x)|
        = 
        \left|
            \frac{y - y_0}{r_n} - \frac{\pi_\Sigma(y) - y_0}{r_n}
        \right|\\
        &= 
        \dist\left(
            \frac{y - y_0}{r_n}, 
            \frac{\Sigma_{y_0,\varepsilon_n} - y_0}{r_n}
        \right)\\ 
        &= 
        \dist\left(
            \frac{y - y_0}{r_n}, 
            T_{y_0}\Sigma \cap B_\varepsilon(0)
        \right)
        + o_{n \to \infty}(1),
    \end{align*}
    where the last equality follows from the equivalence of convergence in the Hausdorff distance and uniform convergence of the distance functions. {  We conclude that for $n$ sufficiently large $\dist(x, \Sigma_{y_0,\varepsilon_n}) \le \delta + o_{n\to\infty}(1)$, so that 
    \[
        \supp\varrho_n \cap {T}_{r_n}^{-1}\left(y_0 + C_{\delta r_n,\varepsilon_n}\right)
        \subseteq 
        \supp\varrho_n \cap 
        \left\{
            \dist\left(\cdot, \Sigma_{y_0,\varepsilon_n}\right) \le \delta + o_{n\to\infty}(1)
        \right\}.
    \]}

    Returning to~\eqref{eq.estimate_Cdelta} with this new inclusion we conclude that
    \begin{equation}\label{eq.estimate3_preview}
        \bar \sigma(C_{\delta,\varepsilon}) 
        \le 
        \liminf_{n \to \infty}
        r_n^{-1}\varrho_0\left(
            \left\{
                \dist\left(\cdot, \Sigma_{y_0,\varepsilon_n}\right) \le 2\delta
            \right\}
        \right). 
    \end{equation}

    To conclude we estimate the right-hand side above. We claim that 
    \begin{equation}\label{eq.volume_estimate}
        \varrho_n\left(
            \left\{
                \dist\left(\cdot, \Sigma_{y_0,\varepsilon_n}\right) \le \delta
            \right\}
        \right)
        \le 
        \norm{\varrho_0}_\infty 
        \omega_{d-1}\mathscr{H}^1(\Sigma_{y_0,\varepsilon_n}) 
        {\delta}^{d-1}
        + o(\varepsilon_n).
    \end{equation}
    This estimate will be proven with a slight refinement of the induction strategy from~\cite[Lemma 4.2]{mosconi2005gamma}. 

    First recall that by the construction from Lemma~\ref{lemma.noncut_property}, both $\Sigma_{y_0,\varepsilon_n}$ and $\Sigma\setminus \Sigma_{y_0,\varepsilon_n}$ are connected and we have that
    \[
        \Sigma_{y_0,\varepsilon_n} \cap \partial B_{\varepsilon_n}(y_0)
        = 
        \{y_{n,1}, y_{n,2}\}
        =
        \Sigma_n \cap \partial B_{\varepsilon_n}(y_0).
    \]
    In particular, $\Sigma_{y_0,\varepsilon_n}$ is $1$-rectifiable and can be covered by countably many connected sets ${\left(\Xi_{n,i}\right)}_{i \in \mathbb{N}}$. We assume without loss of generality that: 
    \begin{itemize}
        \item $\Xi_{n,1} \subset \Gamma$;
        \item $\Xi_{n,1}$ contains the two points of $\Sigma_{y_0,\varepsilon_n}$ on the boundary $\partial B_{\varepsilon_n}(y_0)$;
        \item and as a consequence $\mathscr{H}^1(\Xi_{n,1}) \ge 2\varepsilon_n$.
    \end{itemize}
    In addition, we can assume that the remaining sets $\Xi_{n,i}$ are piece-wise disjoint and for all $i \ge 2$ we have that 
    \[
        \mathscr{H}^1(\Xi_{n,i}) 
        \le 
        \mathscr{H}^1(\Sigma_{y_0,\varepsilon_n}\setminus \Gamma) = o(\varepsilon_n), 
    \]
    where the last equality comes from the blow-up theorem and the fact that $y_0$ is a flat point of both $\Sigma$ and $\Gamma$. 

    First we estimate the volume of the points at distance at most $\delta$ to $\Xi_{n,1}$. Indeed, we can decompose the set 
    \[
        A(\Xi_{n,1}, \delta)
        \eqdef 
        \left\{
            \dist(\cdot, \Xi_{n,1}) \le \delta            
        \right\}
        \subseteq 
        C(\Xi_{n,1}, \delta)
        \cup 
        H(\Xi_{n,1}, \delta),
    \]
    where $C((\Xi_{n,1}, \delta))$ is a tubular region around $\Xi_{n,1}$ and $H(\Xi_{n,1}, \delta)$ is a union of two hemispheres centered at its end-points. For the tubular region we have the bound 
    \[
        \varrho_n\left(C(\Xi_{n,1}, \delta)\right) 
        \le 
        \norm{\varrho_0}_\infty
        \mathscr{H}^1(\Xi_{n,1})\omega_{d-1} {\delta}^{d-1}.
    \]
    On the other hand, for the two hemispheres we have that 
    \[
        \varrho_n\left(H(\Xi_{n,1}, \delta)\right)  
        = 
        o(\varepsilon_n),   
    \]
    since either they are at minimal distance to $\Sigma$ outside of $B_{r_n}$, hence not in the support of $\varrho_n$, or their projection onto $\Sigma$ is contained in $\{y_{n,1}, y_{n,2}\} \cup \Sigma_{y_0,\varepsilon_n}\setminus\Gamma$. Hence 
    \[
        \varrho_n\left(H(\Xi_{n,1}, \delta)\right)  
        \le 
        \nu_n\left( \Sigma_{y_0,\varepsilon_n}\setminus\Gamma\right)
        = 
        o(\varepsilon_n), 
    \]
    and we have proven the first step induction towards~\eqref{eq.volume_estimate}. 

    To finish the proof define 
    \[
        C_k \eqdef \bigcup_{i = 1}^k \Xi_{n,i},
    \]
    assume that~\eqref{eq.volume_estimate} holds with $\Sigma_{y_0,\varepsilon_n}$ replaced by $C_k$ and let us show that it holds for $C_{k+1}$. In this case, we have that 
    \begin{align*}
        \varrho_n(A(C_{k+1}, \delta)) 
        &= 
        \varrho_n(A(C_{k}, \delta))
        + 
        \varrho_n(A(\Xi_{n,k+1}, \delta)\setminus A(C_{k}, \delta))\\
        &\le 
        \norm{\varrho_0}_\infty 
        \left[
            \omega_{d-1}\mathscr{H}^1(C_k) 
            {\delta}^{d-1}
            + 
            |A(\Xi_{n,k+1}, \delta)\setminus A(C_{k}, \delta)|
        \right] 
        + o(\varepsilon_n).
    \end{align*}
    Hence, let us estimate $|A(\Xi_{n,k+1}, \delta)\setminus A(C_{k}, \delta)|$. Once again, $A(\Xi_{n,k+1}, \delta)$ will have one tubular region and two hemispheres, but since $\Xi_{n,k+1}$ touches $C_k$, we can remove at least one ball of volume $\omega_d\delta^d$, which makes up for the two hemispheres. This way we have that 
    \begin{align*}
        \varrho_n(A(C_{k+1}, \delta)) 
        &\le 
        \norm{\varrho_0}_\infty 
        \omega_{d-1} {\delta}^{d-1}
        \left[
            \mathscr{H}^1(C_k) 
            + 
            \mathscr{H}^1(\Xi_{n,k+1}) 
        \right] 
        + o(r_n)\\
        &= 
        \norm{\varrho_0}_\infty 
        \omega_{d-1} {\delta}^{d-1}
        \mathscr{H}^1(C_{k+1}) 
        + o(\varepsilon_n).
    \end{align*}
    As a result, this estimate holds for every $k \in \mathbb{N}$, and since by construction $\mathscr{H}^1(C_k) \to \mathscr{H}^1(\Sigma_{y_0,\varepsilon_n})$, we obtain the bound~\eqref{eq.volume_estimate}. 

    Going back to~\eqref{eq.estimate3_preview}, and recalling that by definition $\frac{\varepsilon_n}{r_n} \to \varepsilon$ we obtain that
    \[
        \bar \sigma(C_{\delta,\varepsilon})
        \le 
        \liminf_{n \to \infty}
        \norm{\varrho_0}_\infty
        \frac{\H^1\left(\Sigma_{y_0,\varepsilon_n}\right)}{r_n}
        \omega_{d-1}{\delta}^{d-1}
        + 
        \frac{o(\varepsilon_n)}{r_n}
        = 
        2\omega_{d-1}\norm{\varrho_0}_\infty \varepsilon \delta^{d-1}.
    \]
    Taking the limit as $\delta \to 0$ and $\varepsilon = 1$, we obtain that $\bar \sigma(T_{y_0}\Sigma\cap \overline{B_1(0)}) = 0$.

    Finally, to prove item $(iv)$, recall the sequences $\sigma_n$ and $\nu_\star\mres B_{r_n}$, and let $\gamma_n$ be the optimal transportation plan between them. From Proposition~\ref{prop.projection_to_loops}, it follows that
    \[
        \supp \gamma_n\mres \mathbb{R}^d\times \Gamma 
        \subset 
        \text{graph}(\Pi_\Sigma),
    \]
    { 
    where $\Pi_\Sigma$ denotes the multivalued projection operator over $\Sigma$, \textit{i.e.}
    \[
        \Pi_\Sigma(x) \eqdef 
        \left\{
            y \in \Sigma: |x-y| = \dist(x,\Sigma)
        \right\}
        \text{ and }
        \text{graph}(\Pi_\Sigma) 
        \eqdef
        \left\{
            (x,y): y \in \Pi_\Sigma(x)
        \right\}.
    \]}
    Since $\bar \sigma_n, \bar \nu_n$ are generated by the push-forward of $\sigma_n$ and $\nu_\star\mres B_{r_n}$ by $\Phi^{y_0, r_n}$, the optimal transportation between them in given by the plan 
    \[
        \bar\gamma_n 
        \eqdef 
        \frac{1}{r_n}
        {\left(\Phi^{(y_0,y_0),r_n}\right)}_\sharp \gamma_n, 
        \text{ so that }
        \supp \left(
            \bar \gamma_n \mres 
            \left(
                \mathbb{R}^d\times \frac{\Gamma - y_0}{r_n} 
            \right)
        \right)
        \subset 
        \text{graph}
        \left(\Pi_{\frac{\Sigma - y_0}{r_n}}\right).
    \]
    If $\Sigmayrn$ was entirely contained in $\Gamma$, the proof would be strictly the same as in the analogous result from~\cite{chambolle2025one}. Here this is not the case, but the set part of $\Sigmayrn$ where the projection property might fail is small since $\mathscr{H}^1(\Sigmayrn \setminus \Gamma) = o(r_n)$.

    Up to a subsequence $\bar\gamma_n$ converges to some $\bar \gamma$, which, by the stability of optimal transportation plans, also transports $\bar \sigma$ to $\bar \nu$ optimally, let us show that $\supp \bar \gamma \subset \text{graph}\left(\Pi_{T_{y_0}\Sigma}\right)$. Notice that
    for any $A \subset \mathbb{R}^d$, we have that 
    \begin{align*}
        \bar\gamma_n\left(
            A\times 
            \left(
                \frac{\Sigmayrn\setminus\Gamma - y_0}{r_n}
            \right)
            \right)
            &\le 
            \frac{1}{r_n} 
            \nu_\star \left(
            \Sigmayrn\setminus\Gamma
            \right)
            = 
            \frac{o(r_n)}{r_n}
            \cvstrong{n \to \infty}{}0,
    \end{align*}
    since $\theta_1(\nu_\star,y_0) < +\infty$ and the tangent spaces of $\Sigma$ and $\Gamma$ coincide at $y_0$, from~\eqref{eq.assumptions_y_0}. 
    
    As a result, given $(x,p) \in \supp \bar \gamma$, there is an open ball $B$ centered at $(x,p)$ such that
    \[
        0 < \bar \gamma(B) \le \liminf_{n \to \infty} \bar\gamma_n(B) = \liminf_{n \to \infty} \bar\gamma_n\left(
            B \cap 
            \left(
                \mathbb{R}^d\times \frac{\Gamma - y_0}{r_n} 
            \right)
        \right) .
    \]
    In particular, we can find $\supp \bar \gamma_n \mres \left(
        \mathbb{R}^d\times \frac{\Gamma - y_0}{r_n} 
    \right) \ni (x_n, p_n) \xrightarrow[n \to \infty]{} (x,p)$. So it holds that
    \[
        |x - p| 
        = 
        \lim_{n \to \infty}   
        |x_n - p_n|
        =  
        \lim_{n \to \infty}   
        \dist\left(x_n, \frac{\Sigma - y_0}{r_n}\right) 
        = 
        \dist(x, T_{y_0}\Sigma),
    \]
    where the last equality comes from the point-wise convergence of the distance functions from Kuratowski convergence of blow-ups from Lemma~\ref{lemma.blowup_domain_measure}. 
\end{proof}

\subsection{Better competitor and absence of loops}\label{subsec.absence_loops}
We now implement Step 5 from Section~\ref{subsection.localization_blowup_arg} obtaining a contradiction to the fact that the optimal set $\Sigma$ contains a loop. Let us recall the construction done so far. {  Let $\Sigma$ be the support of an optimal measure for~\eqref{problem.shape_optimization_relaxed}, assume it contains a loop $\Gamma$ and we can choose a suitable non-cut point $y_0 \in \Gamma$, as in~\eqref{eq.def_y}.} Then we can perform the localizations around $y_0$ from the previous subsection and obtain the measures $\bar \sigma$ and $\bar \nu$, as in~\eqref{eq.limit_blowup_measures}. From Lemma~\ref{lemma.properties_Gamma_limit}, the latter is a minimizer of the functional $F$ defined in~\eqref{eq.limit_F} and
\[
    \bar \nu = \theta \H^1\mres T_{y_0}\Sigma \in \argmin F, 
    \text{ where }
    \theta = 2\theta_1(\nu_\star, y_0).
\]
As the optimal transportation from $\bar \sigma$ to $\bar \nu$ is attained by the projection map onto $T_{y_0}\Sigma$, we use a refined version of the argument done in~\cite[Lemma 6.3]{chambolle2025one} to construct a strictly better competitor to $F$. The further complexity of this case stems from the fact that we must remove all the mass of a small segment and create an advantageous structure, see Figure~\ref{figure.loops_competitor}. This construction will then contradict the existence of loops, so that any optimal $\Sigma$ must be a tree.  
\begin{figure}[t]
    \centering
    \tikzset{every picture/.style={line width=0.75pt}} 

\begin{tikzpicture}[x=0.75pt,y=0.75pt,yscale=-1,xscale=1]

\draw   (344,136.4) .. controls (344,82.61) and (387.61,39) .. (441.4,39) .. controls (495.19,39) and (538.8,82.61) .. (538.8,136.4) .. controls (538.8,190.19) and (495.19,233.8) .. (441.4,233.8) .. controls (387.61,233.8) and (344,190.19) .. (344,136.4) -- cycle ;
\draw  [fill={rgb, 255:red, 208; green, 2; blue, 27 }  ,fill opacity=1 ] (437,136.4) .. controls (437,133.91) and (438.97,131.9) .. (441.4,131.9) .. controls (443.83,131.9) and (445.8,133.91) .. (445.8,136.4) .. controls (445.8,138.89) and (443.83,140.9) .. (441.4,140.9) .. controls (438.97,140.9) and (437,138.89) .. (437,136.4) -- cycle ;
\draw    (441.4,136.4) ;
\draw  [dash pattern={on 4.5pt off 4.5pt}]  (389.8,23) -- (576.8,215) ;
\draw  [dash pattern={on 4.5pt off 4.5pt}]  (310.8,192) -- (500.8,20) ;
\draw  [dash pattern={on 4.5pt off 4.5pt}]  (321.8,85) -- (502.8,269) ;
\draw  [dash pattern={on 4.5pt off 4.5pt}]  (376.8,264) -- (578.8,84) ;
\draw [color={rgb, 255:red, 0; green, 0; blue, 0 }  ,draw opacity=1 ]   (441.4,136.4) -- (415.8,160) ;
\draw    (406.4,169.4) -- (392.99,181.95) ;
\draw [shift={(390.8,184)}, rotate = 316.9] [fill={rgb, 255:red, 0; green, 0; blue, 0 }  ][line width=0.08]  [draw opacity=0] (8.93,-4.29) -- (0,0) -- (8.93,4.29) -- cycle    ;
\draw    (473.8,178) -- (485.75,190.81) ;
\draw [shift={(487.8,193)}, rotate = 226.97] [fill={rgb, 255:red, 0; green, 0; blue, 0 }  ][line width=0.08]  [draw opacity=0] (8.93,-4.29) -- (0,0) -- (8.93,4.29) -- cycle    ;
\draw    (471.4,106.4) -- (484.68,93.12) ;
\draw [shift={(486.8,91)}, rotate = 135] [fill={rgb, 255:red, 0; green, 0; blue, 0 }  ][line width=0.08]  [draw opacity=0] (8.93,-4.29) -- (0,0) -- (8.93,4.29) -- cycle    ;
\draw    (405.8,106) -- (393.85,93.19) ;
\draw [shift={(391.8,91)}, rotate = 46.97] [fill={rgb, 255:red, 0; green, 0; blue, 0 }  ][line width=0.08]  [draw opacity=0] (8.93,-4.29) -- (0,0) -- (8.93,4.29) -- cycle    ;
\draw    (425.4,137) -- (441.4,136.4) ;
\draw    (441.4,136.4) -- (441.2,159.6) ;
\draw    (441.4,136.4) -- (462.2,137.6) ;
\draw    (441.2,118.6) -- (441.4,136.4) ;
\draw    (441.4,136.4) -- (461.2,118) ;
\draw [fill={rgb, 255:red, 208; green, 2; blue, 27 }  ,fill opacity=1 ]   (76.01,33.1) -- (76.01,88.54) ;
\draw [shift={(76.01,33.1)}, rotate = 270] [color={rgb, 255:red, 0; green, 0; blue, 0 }  ][line width=0.75]    (0,5.59) -- (0,-5.59)   ;
\draw    (76.01,195.38) -- (76.01,250.83) ;
\draw [shift={(76.01,250.83)}, rotate = 270] [color={rgb, 255:red, 0; green, 0; blue, 0 }  ][line width=0.75]    (0,5.59) -- (0,-5.59)   ;
\draw [color={rgb, 255:red, 208; green, 2; blue, 27 }  ,draw opacity=1 ] [dash pattern={on 0.84pt off 2.51pt}]  (76.01,88.54) -- (76.01,177.12) -- (76.01,195.38) ;
\draw [shift={(76.01,195.38)}, rotate = 270] [color={rgb, 255:red, 208; green, 2; blue, 27 }  ,draw opacity=1 ][line width=0.75]    (0,5.59) -- (0,-5.59)   ;
\draw [shift={(76.01,88.54)}, rotate = 270] [color={rgb, 255:red, 208; green, 2; blue, 27 }  ,draw opacity=1 ][line width=0.75]    (0,5.59) -- (0,-5.59)   ;
\draw [color={rgb, 255:red, 208; green, 2; blue, 27 }  ,draw opacity=1 ] [dash pattern={on 0.84pt off 2.51pt}]  (285.97,127.5) -- (78.78,127.5) ;
\draw [shift={(76.78,127.5)}, rotate = 360] [fill={rgb, 255:red, 208; green, 2; blue, 27 }  ,fill opacity=1 ][line width=0.08]  [draw opacity=0] (12,-3) -- (0,0) -- (12,3) -- cycle    ;
\draw [color={rgb, 255:red, 208; green, 2; blue, 27 }  ,draw opacity=1 ] [dash pattern={on 0.84pt off 2.51pt}]  (286.2,159.12) -- (79.01,159.12) ;
\draw [shift={(77.01,159.12)}, rotate = 360] [fill={rgb, 255:red, 208; green, 2; blue, 27 }  ,fill opacity=1 ][line width=0.08]  [draw opacity=0] (12,-3) -- (0,0) -- (12,3) -- cycle    ;
\draw [fill={rgb, 255:red, 208; green, 2; blue, 27 }  ,fill opacity=1 ] [dash pattern={on 4.5pt off 4.5pt}]  (204.64,33.1) -- (204.77,250.19) ;
\draw  [dash pattern={on 4.5pt off 4.5pt}]  (78.45,257.37) -- (203.25,257.37) ;
\draw [shift={(205.25,257.37)}, rotate = 180] [color={rgb, 255:red, 0; green, 0; blue, 0 }  ][line width=0.75]    (10.93,-3.29) .. controls (6.95,-1.4) and (3.31,-0.3) .. (0,0) .. controls (3.31,0.3) and (6.95,1.4) .. (10.93,3.29)   ;
\draw [shift={(76.45,257.37)}, rotate = 0] [color={rgb, 255:red, 0; green, 0; blue, 0 }  ][line width=0.75]    (10.93,-3.29) .. controls (6.95,-1.4) and (3.31,-0.3) .. (0,0) .. controls (3.31,0.3) and (6.95,1.4) .. (10.93,3.29)   ;
\draw [color={rgb, 255:red, 208; green, 2; blue, 27 }  ,draw opacity=1 ]   (76.01,88.54) -- (129.6,88.8) ;
\draw [color={rgb, 255:red, 208; green, 2; blue, 27 }  ,draw opacity=1 ] [dash pattern={on 0.84pt off 2.51pt}]  (77.76,89.52) -- (203.6,159.8) ;
\draw [shift={(76.01,88.54)}, rotate = 29.18] [fill={rgb, 255:red, 208; green, 2; blue, 27 }  ,fill opacity=1 ][line width=0.08]  [draw opacity=0] (12,-3) -- (0,0) -- (12,3) -- cycle    ;
\draw [color={rgb, 255:red, 208; green, 2; blue, 27 }  ,draw opacity=1 ] [dash pattern={on 0.84pt off 2.51pt}]  (99.49,90.44) -- (203.6,125.8) ;
\draw [shift={(97.6,89.8)}, rotate = 18.76] [fill={rgb, 255:red, 208; green, 2; blue, 27 }  ,fill opacity=1 ][line width=0.08]  [draw opacity=0] (12,-3) -- (0,0) -- (12,3) -- cycle    ;
\draw   (101.2,87.8) .. controls (101.2,84.51) and (99.55,82.86) .. (96.26,82.86) -- (96.26,82.86) .. controls (91.55,82.86) and (89.2,81.21) .. (89.2,77.92) .. controls (89.2,81.21) and (86.85,82.86) .. (82.14,82.86)(84.26,82.86) -- (82.14,82.86) .. controls (78.85,82.86) and (77.2,84.51) .. (77.2,87.8) ;
\draw   (128.2,64.6) .. controls (128.27,59.93) and (125.98,57.56) .. (121.31,57.49) -- (113.43,57.37) .. controls (106.76,57.26) and (103.47,54.88) .. (103.54,50.21) .. controls (103.47,54.88) and (100.1,57.16) .. (93.43,57.05)(96.43,57.1) -- (84.31,56.91) .. controls (79.64,56.84) and (77.27,59.13) .. (77.2,63.8) ;

\draw (443.4,139.8) node [anchor=north west][inner sep=0.75pt]  [font=\footnotesize]  {$\mathbb{R} e_{d}$};
\draw (306,211.4) node [anchor=north west][inner sep=0.75pt]  [font=\scriptsize]  {$ \begin{array}{l}
\ \ \ \ \ \ \overline{\theta }_{i}  >0\\
\overline{\sigma }_{i}\left( C_{\varepsilon }^{i}\right) \approx \overline{\theta }_{i} \varepsilon 
\end{array}$};
\draw (400.6,180.1) node [anchor=north west][inner sep=0.75pt]  [font=\footnotesize,color={rgb, 255:red, 0; green, 0; blue, 0 }  ,opacity=1 ]  {$\xi _{i}$};
\draw (496,215.4) node [anchor=north west][inner sep=0.75pt]  [font=\scriptsize]  {$ \begin{array}{l}
\ \ \ \ \ \ \overline{\theta }_{j} =0\\
\overline{\sigma }_{j}\left( C_{\varepsilon }^{j}\right) =o( \varepsilon )
\end{array}$};
\draw (514,31.4) node [anchor=north west][inner sep=0.75pt]  [font=\scriptsize]  {$ \begin{array}{l}
\ \ \ \ \ \ \overline{\theta }_{i'}  >0\\
\overline{\sigma }_{i'}\left( C_{\varepsilon }^{i'}\right) \approx \overline{\theta }_{i'} \varepsilon 
\end{array}$};
\draw (298,25.4) node [anchor=north west][inner sep=0.75pt]  [font=\scriptsize]  {$ \begin{array}{l}
\ \ \ \ \ \ \overline{\theta }_{j'} =0\\
\overline{\sigma }_{j'}\left( C_{\varepsilon }^{j'}\right) =o( \varepsilon )
\end{array}$};
\draw (13.2,38.12) node [anchor=north west][inner sep=0.75pt]  [font=\footnotesize]  {$[ -e_{d} ,e_{d}]$};
\draw (126.65,261.44) node [anchor=north west][inner sep=0.75pt]  [font=\footnotesize]  {$L/2$};
\draw (45.98,121.07) node [anchor=north west][inner sep=0.75pt]  [font=\scriptsize,color={rgb, 255:red, 74; green, 144; blue, 226 }  ,opacity=1 ]  {$\overline{s} +t$};
\draw (44.76,152.34) node [anchor=north west][inner sep=0.75pt]  [font=\scriptsize,color={rgb, 255:red, 74; green, 144; blue, 226 }  ,opacity=1 ]  {$\overline{s} -t$};
\draw (223.5,34.99) node [anchor=north west][inner sep=0.75pt]  [font=\scriptsize]  {$\overline{\theta }_{i}  >0$};
\draw (63.49,136.23) node [anchor=north west][inner sep=0.75pt]  [font=\scriptsize,color={rgb, 255:red, 74; green, 144; blue, 226 }  ,opacity=1 ]  {$\overline{s}$};
\draw (78.01,64.22) node [anchor=north west][inner sep=0.75pt]  [font=\scriptsize]  {$\ell _{i}( t)$};
\draw (92,35.4) node [anchor=north west][inner sep=0.75pt]  [font=\scriptsize]  {$\ell _{i}( \varepsilon )$};

\end{tikzpicture}
    \caption{Construction of a better competitor in Theorem~\ref{theorem.no_loops_integrable_regime}. On the right, the partition of the space into sections. For sections $i,i'$ such that $\bar \theta_i,\bar \theta_{i'}>0$ we add a segment in their direction. For $\bar \theta_j,\bar \theta_{j'}=0$ we construct a Dirac mass. On the cases of positive density we have a gain of order $\varepsilon^2$ in transportation cost, for zero density we lose $o(\varepsilon^2)$. On the left the transportation strategy of each section of the partitioned space.}\label{figure.loops_competitor}
\end{figure}
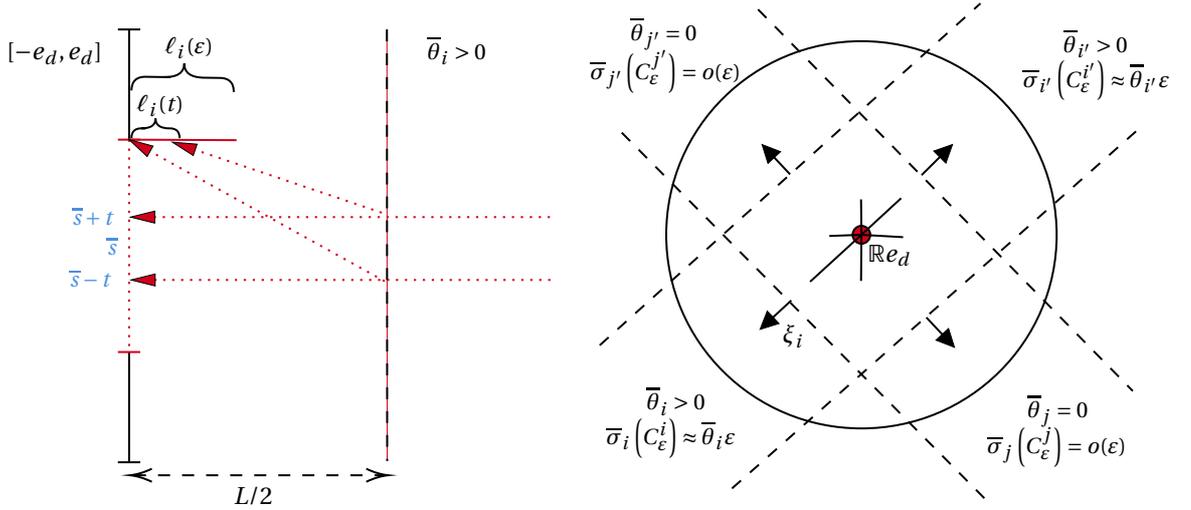

\begin{theorem}\label{theorem.no_loops_integrable_regime}
    { 
    Assume that $\varrho_0$ is an atomic measure, as in Case 1. Then the support $\Sigma$ of any solution to the relaxed problem~\eqref{problem.shape_optimization} is a tree, in the sense that it does not contain homeomorphic images of $\mathbb{S}^1$.
    
    In the Case 2, that is $\varrho_0 \in L^\infty(\Omega)$ is given by a bounded density with compact support, assume that either $d\ge 3$ and $p \ge 1$, or that $d = 2$ and $p>1$. Then the original problem~\eqref{problem.shape_optimization} admits solutions $\Sigma$, and any such optimal network is a tree.

    In Case 3, where $\varrho_0$ is a mixture between a bounded density and an atomic measure, if either $d\ge 3$ and $p \ge 1$ or if $d = 2$ and $p>1$, the support of any solution to the relaxed problem is a tree. }
\end{theorem}
\begin{proof}
    Suppose by contradiction that $\Sigma$ is optimal and contains a loop, and let $y_0$ be a flat non-cut point inside this loop, chosen as in~\eqref{eq.def_y}. Up to a rotation, we may assume that $T_{y_0}\Sigma = \mathbb{R}^d e_d$, where ${(e_i)}_{i = 1}^d$ is a basis of ${\mathbb{R}}^d$. We will start with a simpler construction for Case 1; and then use it as a building block for the second one. 
    
    \textbf{\underline{Case 1:}} Recall that $\supp \bar\sigma \subset
    \left\{x = (x',x_d) \in 
    {\mathbb{R}}^d : |x'|> L,\  |x_d|\le 1 \right\}$, as shown in item (i) of Lemma~\ref{lemma.properties_Gamma_limit},
    so we can cover its support with finitely many sets ${(E_i)}_{i=1}^N$ defined as:
    \[
        E_i
        \eqdef 
        \left\{
            x = (x',x_d) \in {\mathbb{R}}^d : 
                 \inner{\xi_i, x} > L/2, 
                 \  |x_d| \le 1
        \right\}
    \]
    where $\xi_i \in \mathbb{S}^{d-1} \cap {[e_d]}^\perp$ are unit
    vectors and $N$ depends only on the dimension.
    We then define a disjoint family
    \[
        F_1 = E_1, \quad F_{i+1} = E_{i+1}\setminus \bigcup_{j=1}^i F_i \text{ for $i\ge 1$}
    \]
    and decompose our measures $\bar \sigma$ and $\bar \nu$ as 
    \[
        \bar \sigma  = \sum_{i = 1}^N \bar \sigma_i, \ 
        \bar \nu     = \sum_{i = 1}^N \bar \nu_i
        \text{ where } 
        \bar \sigma_i  \eqdef \bar \sigma \mres F_i 
        \text{ and }
        \bar \nu_i \eqdef {(\proj_d)}_\sharp\bar \sigma_i,
    \]
    where $\proj_d : x\mapsto x_d e_d$ is the projection onto the vertical axis.
    By Besicovitch's differentiation theorem, $\bar\nu_i = \theta_i \H^1\mres[-e_d, e_d]$, where $\theta_i(s) = \theta_i(s e_d)\ge 0$ sum up to a positive constant 
    \[
        \sum_{i = 1}^N \theta_i(s) = \theta > 0.    
    \]
    
    In the sequel, we introduce the  notation: $\mathbb{R}^d \ni x = (x_i, x''_i, x_d)$ where $x_i = \inner{\xi_i,x}$ is the component of $x$ parallel to $\xi_i$ and $x''_i \in {[\xi_i, e_d]}^{\perp}$. Defining the sets 
    \begin{equation}\label{eq.box_set_covering}
        C_{t}^i 
        \eqdef 
        F_i \cap \{ x \in \mathbb{R}^d :  |x_d - \bar s|\le t \}
        \subset \left\{
            x = (x_i, x''_i, x_d) : 
            \substack{
                \displaystyle
                x_i > L/2, \\
                \displaystyle
                |x_d - \bar s|\le t
            }
        \right\},
    \end{equation}
    and letting $\bar s\in (-1,1)$ be a common Lebesgue point of all $\theta_i$,
    $i=1,\dots, N$, it follows from the fact that ${(\proj_d)}_\sharp \bar \sigma_i = \theta_i\H^1\mres [-e_d,e_d]$ that, for every $i = 1,\dots,N$
    \begin{equation}
        \frac{\bar \sigma_i(C_\varepsilon^i)}{2\varepsilon} = 
        \frac{1}{2\varepsilon} \int_{\bar s-\varepsilon}^{\bar s+\varepsilon}\theta_i(t)\dd t 
        \xrightarrow[\varepsilon \to 0]{} 
        \theta_i(\bar s).
    \end{equation}

    Consider now the two subfamilies of indexes
    \begin{equation}\label{eq.indexes_positive_mass}
        I_1
        \eqdef 
        \{i : \ \theta_i(\bar s) > 0\}, \quad 
        I_2
        \eqdef 
        \{i : \ \theta_i(\bar s) = 0\}. 
    \end{equation}
    In particular, for each $i \in I_1$, there is a constant $\bar \theta_i > 0$ and $\varepsilon > 0$ such that for $t < \varepsilon$ we have
    \begin{equation}\label{eq.condition_t_noloop}
        \frac{1}{\bar\theta_i}
        \le 
        \frac{\bar \sigma_i(C_{t}^i)}{t}
        \le 
        \bar\theta_i.
    \end{equation}

    Now let us exploit the fact that, from Lemma~\ref{lemma.properties_Gamma_limit} the optimal transport is given by projections to propose a new transportation map, sending the mass in $C_{\varepsilon}^i$ to a segment pointing towards $\xi_i$:
    \[
        \bar T(x)
        \eqdef 
        \begin{cases}
            \ell_i(|x_d-\bar s|)\xi_i + (\bar s + \varepsilon)e_d,& 
            \text{ if $x \in C_{\varepsilon}^i$ and $i \in I_1$},\\ 
            (\bar s + \varepsilon)e_d, & 
            \text{ if $x \in C_{\varepsilon}^i$ and $i \in I_2$,}\\
            \proj_d(x),& \text{ otherwise,}
        \end{cases}
    \]
    where, for each $i \in I_1$, we define $\ell_i:[0,\varepsilon] \to \mathbb{R}_+$ via the conservation of mass relation 
    \begin{equation}\label{eq:conservmass_noloop}
        \ell_i(t) = \alpha \bar\sigma_i(C_{t}^i),
    \end{equation}
    {  where we recall $\alpha = \mathcal{L}(\nu_\star)$.} In other words, the mass that was sent to the vertical segment $[\bar s-\varepsilon', \bar s+\varepsilon']e_d$ is now used to form the horizontal segments
    \[
        L_i \eqdef (\bar s + \varepsilon) e_d  +[0, \ell_i(\varepsilon)]\xi_i,
    \]
    for each $i\in I_1$. The mass corresponding to the remaining indexes form a Dirac measure concentrated in $(\bar s + \varepsilon) e_d$, but with a mass of order $o(\varepsilon)$. 
   
    Thanks to~\eqref{eq:conservmass_noloop}, the map $\bar T$ sends $\bar \sigma_i\mres C_\varepsilon^i$ to
    the measure $\alpha^{-1}\H^1\mres L_i$, hence the transported measure $\bar T_\sharp \bar \sigma$ satisfies the constraints in the definition~\eqref{eq.limit_F} of the limiting functional $F$, since the newly added structure, given by 
    \[
        \Sigma' = \bigcup_{i \in I_1} L_i, 
    \]
    is a connected set. As a result, one has that $F(\bar T_\sharp \bar\sigma)<+\infty$. 
    
    { 
    So for $i \in I_1$ and $x \in C_\varepsilon^i$, using the inclusion~\eqref{eq.box_set_covering}, the upper and lower bounds on $\bar\sigma_i(C^i_t)$ from~\eqref{eq.condition_t_noloop}, and recalling the notation $x = (x_i, x''_i, x_d)$ we have that}
    \begin{align*}
        |x - \proj_d(x)|^2 - |x-  \bar T(x)|^2 
        &= 
        x_i^2 + |x''_i|^2 - 
        {(x_i-\ell_i(|x_d-\bar s|))}^2 - |x''_i|^2 - {(x_d- \bar s - \varepsilon)}^2
        \\
        &=
        2x_i\ell_i(|x_d-\bar s|) - {\ell_i(|x_d-\bar s|)}^2 - {(x_d-\bar s)}^2
        + 2\varepsilon|x_d - \bar s| - \varepsilon^2
        \\ 
        &\ge 
        \frac{L\alpha}{\bar \theta_i}|x_d - \bar s| 
        - 
        (\alpha \bar\theta_i |x_d - \bar s|)^2 
        +
        2\varepsilon|x_d - \bar s| 
        - 
        \varepsilon^2\\ 
        &\ge
        \left(
            \frac{L\alpha}{\bar \theta_i} + 2\varepsilon
        \right)
        |x_d - \bar s| 
        - 
        \left(
            2 + \alpha^2 \bar \theta_i^2
        \right)\varepsilon^2. 
    \end{align*}
    
    This is a {  quantitative} estimate on the difference of the squared distance, to extend it to the $p$-power, we use that for any $a,b > 0$ {  taken on a compact set we have}
    \begin{equation}\label{eq.taylor}
        a^{p/2} - b^{p/2} = \frac{p}{2}b^{\frac{p}{2}-1}(a - b) + o(a - b),
    \end{equation}
    {  where the $o(\cdot)$ depends on the compact. So our goal is to take $a = |x - \proj_d(x)|^2$ and $b = |x - \bar T(x)|^2$, since 
    $|x - \bar T(x)|^2 \ge {(x_i - \ell_i(|x_d - \bar s|))}^2$ and over $C^i_\varepsilon$ we have $x_i \ge L/2$ and $\ell_i(|x_d - \bar s|) \le \alpha \bar \theta_i \varepsilon$, taking $\varepsilon$ small enough we have $|x - \bar T(x)| \ge L/4$. Since $a,b$ taken as above are uniformly bounded since the transport takes place inside a compact set, there exists a constant $K$ depending on $p,\bar\theta_i$ and $\alpha$ such that
    \begin{align*}
        |x - \proj_d(x)|^p - |x - \bar T(x)|^p 
        &\ge
        LK\left(
            |x - \proj_d(x)|^2 - |x - \bar T(x)|^2
        \right)
        + o(\varepsilon)\\ 
        &\ge 
        LK|x_d - \bar s|
        + o(\varepsilon),
    \end{align*}
    where we have replaced $o(a-b)$ with a $o(\varepsilon)$ since for $x \in C^i_\varepsilon$, $\proj_d(x)$ and $\bar T(x)$ are both sent to the same interval of $\mathbb{R}e_d$ with length $2\varepsilon$ so that $a - b = O(\varepsilon)$. 
    }
    
    Next, we notice that there exists $n_i \in \mathbb{N}$, to be fixed later, such that for any $x \in C_\varepsilon^i \setminus C_{\frac{\varepsilon}{n_i}}$ we have that $|x_d - \bar s| \ge \frac{\varepsilon}{n_i}$. Hence, integrating with respect to $\bar \sigma_i$ over $C_\varepsilon^i$, and using~\eqref{eq.condition_t_noloop} to estimate the integral over $C_{\frac{\varepsilon}{n_i}}$ we have in particular that $\bar \sigma(C_{\frac{\varepsilon}{n_i}}) \gtrsim \varepsilon$, yielding
    \begin{align*}
        \int_{C_\varepsilon^i}
        &
        \left(|x - \proj_d(x)|^p - |x - \bar T(x)|^p \right)
        \dd \bar \sigma_i
        \ge 
        {  LK} \int_{C_\varepsilon^i \setminus C^i_{\frac{\varepsilon}{n_i}}}|x_d - \bar s|\dd \bar \sigma_i +  o(\varepsilon^2) \\ 
        &\ge 
        {  LK}\frac{\varepsilon}{n_i}
        \bar \sigma_i
        \left(
            C_\varepsilon^i \setminus C_{\frac{\varepsilon}{n_i}}^i
        \right)
        +  
        o(\varepsilon^2)
        =
        {  LK}\frac{\varepsilon}{n_i}
        \left(
            \bar \sigma_i
            \left(
                C_\varepsilon^i
            \right)
            - 
            \bar \sigma_i
            \left(
                C_{\frac{\varepsilon}{n_i}}^i
            \right)
        \right)
        +  
        o(\varepsilon^2)\\ 
        &\ge 
        \frac{{  LK}}{n_i}
        \left(
            \frac{1}{\bar \theta_i}
            - 
            \frac{\bar \theta_i}{n_i}
        \right)
        \varepsilon^2
        +  
        o(\varepsilon^2)
        \ge 
        \frac{{  LK}}{2\bar \theta_i n_i}
        \varepsilon^2
        +  
        o(\varepsilon^2),
    \end{align*}
    where in the last inequality we choose $n_i \ge 2\bar \theta_i^2$. 

    For the indexes $i \in I_2$, we observe that the error committed by using the map $\bar T$ is given by $|x - \proj_d(x)|^2 - |x - \bar s e_d|^2 = - {(x_d - \bar s)}^2 \ge - \varepsilon^2$. So using once again~\eqref{eq.taylor} we get that 
    \[
        |x - \proj_d(x)|^p - |x - \bar s e_d|^p
        \ge 
        -{  LK}\varepsilon^2 + o(\varepsilon^2).  
    \]
    Now setting $\nu' \eqdef \bar T_\sharp \bar \sigma$, we obtain that 
    \begin{align*}
        W_p^p(\bar \sigma, \bar \nu) 
        - 
        W_p^p(\bar \sigma, \nu') 
        &\ge
        \int\left(
            |x - \proj_d(x)|^p - |x - \bar T(x)|^p
        \right)\dd \bar \sigma\\
        &= 
        \sum_{i = 1}^N 
        \int_{C_\varepsilon^i}
        \left(
            |x - \proj_d(x)|^p - |x - \bar T(x)|^p
        \right)\dd \bar \sigma_i\\ 
        &\ge 
        {  LK}
        \left(
            \sum_{i \in I_1} 
            \left(
                \frac{1}{2\bar \theta_i n_i} \varepsilon^2 + o(\varepsilon^2)
            \right)
            - 
            \sum_{i \in I_2} 
            (\varepsilon^2 + o(\varepsilon^2)) 
            \bar \sigma_i(C_\varepsilon^i)
        \right)\\ 
        &= 
        {  LK}\varepsilon^2
        \left(
            \sum_{i \in I_1} 
            \left(
                \frac{1}{2\bar \theta_i n_i} + \frac{o(\varepsilon^2)}{\varepsilon^2}
            \right) 
            - 
            \sum_{i \in I_2} 
            \left(
                1 + \frac{o(\varepsilon^2)}{\varepsilon^2}
            \right) 
            \bar \sigma_i(C_\varepsilon^i)
        \right).
    \end{align*}
    The last quantity must be positive for $\varepsilon$ large enough since $\bar\sigma_i(C_\varepsilon^i) = o(\varepsilon)$, for each $i \in I_2$. But as the new competitor $\nu'$ is admissible for the minimization of $F$, we obtain a contradiction with the fact that $\bar \nu$ is a minimizer from Lemma~\ref{lemma.properties_Gamma_limit}. This contradicts the entire construction, meaning that $\Sigma$ does not contain a loop.

    The proof of Case 3 is a combination of the argument above and the strictly stronger argument of Case 2 that is described below. Since the adaptation is quite direct by taking $\delta$ and $\varepsilon$ sufficiently small, depending on the distribution of atoms ${(x_i)}_i$ and the point $y_0$ around which we perform the blow-up construction. Therefore, we only provide the details of the second case.  

    \textbf{\underline{Case 2:}} {  This second case is more delicate since the support of $\bar \sigma$ is no longer from a minimal distance $L$ from the tangent space $T_x\Sigma$. However, from Lemma~\ref{lemma.properties_Gamma_limit} we know have explicit estimates on its mass over cylinders around it, that is 
    \[
        \bar\sigma(C_{\delta,\varepsilon}) \le C \varepsilon \delta^{d-1},
    \]
    where as in Lemma~\ref{lemma.properties_Gamma_limit} the set $C_{\delta,\varepsilon}$ is defined as
    \[
        C_{\delta,\varepsilon}
        \eqdef
        \left\{
            x = (x',x_d) \in {\mathbb{R}}^d : 
            \dist(x, T_{y_0}\Sigma) \le \delta,
            |x_d| \le \varepsilon
        \right\},
    \]
    and $C_{\delta} \eqdef C_{\delta,1}$.
    }

    Next, we perform a similar construction from the one in the previous case, but this time we define
    \[
        E_i
        \eqdef 
        \left\{
            x = (x',x_d) \in {\mathbb{R}}^d : 
                 \inner{\xi_i, x} > \delta/2, 
                 \  |x_d| \le 1
        \right\} 
        \setminus C_\delta.
    \]
    As in Case 1 we can define 
    \[
        F_1 \eqdef E_1, \quad  
        F_i = E_i\setminus F_{i-1}, \quad 
        F_0 \eqdef \mathbb{R}^d\setminus \left(\bigcup_{i = 1}^N F_i\right)
    \]
    and the measures 
    \[
        \bar \sigma_i \eqdef \bar \sigma \mres F_i, \quad 
        \bar \nu_i \eqdef {\left[\proj_d\right]}_\sharp \bar \sigma_i, \quad
        \text{ for } i = 0,\dots, N
    \]
    so that in particular we have that $\displaystyle \bar \sigma = \sum_{i = 0}^N \bar \sigma_i$ and $\displaystyle \bar \nu = \sum_{i = 0}^N \bar \nu_i$. In particular, each $\bar \nu_i$ is rectifiable being written as $\bar \nu_i = \theta_i \H^1\mres [-e_d, e_d]$, and it holds that $\displaystyle \sum_{i = 0}^N \theta_i = \theta$.

    One again, we consider a Lebesgue point $\bar s$ of all densities $\theta_i$ and $\varepsilon_0$ small enough so that for any $\varepsilon < \varepsilon_0$ the equivalent of~\eqref{eq.condition_t_noloop} holds for all $i = 1,\dots,N$. We also recall the sets of indexes $I_1$ and $I_2$ from~\eqref{eq.indexes_positive_mass}, distinguishing the ones with positive density, $\theta_i > 0$ for $i \in I_1$ and $\theta_i = 0$ for $i \in I_2$. {  Since $\bar s$ can be chosen as close to $0$ as we want, to simplify our argument  we assume without loss of generality that $\bar s = 0$.}

    { 
    Finally, we construct the better competitor. For the indexes $i \in I_1$, we construct a segment $L_i$, while for the indexes $i \in I_2$ we send the mass of $\bar \sigma_i$ to a Dirac measure, as in Case 1. The remaining mass inside the cylinder $C_\delta$ also sent to this Dirac measure concentrated at 
    \[
        y_\varepsilon \eqdef \varepsilon e_d, \text{ with total mass } 
        m_\varepsilon \eqdef \bar \sigma (C_{\varepsilon,\delta}) + \sum_{i \in I_2} \bar \sigma_i(C^i_\varepsilon). 
    \]
    Accounting for all these conditions, the new transportation map becomes
    \[
        \bar T(x)
        \eqdef 
        \begin{cases}
            \ell_i(|x_d-\bar s|)\xi_i + (\bar s + \varepsilon)e_d,& 
            \text{ if $x \in C_{\varepsilon}^i$ and $i \in I_1$},\\ 
            y_\varepsilon, & 
            \text{ if $x \in C_{\varepsilon}^i$ and $i \in I_2$, or if $x \in C_{\delta,\varepsilon}$}\\
            \proj_d(x),& \text{ otherwise.}
        \end{cases}
    \]

    Notice that, in order for the construction from Case 1 to work, we need $\delta$ to be larger that at least $2\varepsilon$. This will be enforced by our choice of $\delta$ later on. The new competitor then becomes
    \[
        \nu' 
        \eqdef 
        m_\varepsilon \delta_{y_\varepsilon}
        +
        \sum_{i = 1}^N\alpha^{-1}\H^1\mres L_{i}.
    \]
    As a result, we can estimate the gain in transportation distance by dividing the difference in transportation as 
    \begin{equation}\label{eq.case2_decomposition_trans_gain}
        W_p^p(\bar \sigma, \bar \nu) -  W_p^p(\bar \sigma, \nu')
        =
        {(\Delta W)}_1 + {(\Delta W)}_2 + {(\Delta W)}_3,
    \end{equation}
    where setting $\Delta_p(x) \eqdef |x - \proj_d(x)|^p - |x - \bar T(x)|^p$, each term $\Delta_i$ is defined as 
    \begin{align*}
        {(\Delta W)}_1 
        \eqdef 
        \sum_{i \in I_1} \int_{C_\varepsilon^i} \Delta_p(x) \dd \bar \sigma, \quad
        {(\Delta W)}_2 \eqdef
        \sum_{i \in I_2} \int_{C_\varepsilon^i} \Delta_p(x) \dd \bar \sigma, \quad
        {(\Delta W)}_3 \eqdef
        \int_{C_{\delta,\varepsilon}} \Delta_p(x) \dd \bar \sigma.
    \end{align*}
    In the sequel, we estimate each of these terms similarly to Case 1,
    so we skip some details in the computations.
    
    For $i \in I_1$ and $x\in C_\varepsilon^i$ we can use the same computation to estimate the difference between the quadratic distances:
    \begin{equation}\label{eq.case2_I1_D2_bound}
        \Delta_2(x) = 
        |x - \proj_d(x)|^2 - |x-  \bar T(x)|^2 
        \ge
        C_1\delta
        |x_d - \bar s| 
        - 
        C_2\varepsilon^2,
    \end{equation}
    for some constants $C_1,C_2 > 0$ which depend only on $\alpha$ and $\bar \theta_i$. The difference is the the dependence on $L$ is now replaced by a linear term in $\delta$. To get a bound on $D_p(x)$, we use a slightly different strategy from Case 1. Since now the constant $L$ becomes $\delta$ which will also go to 0, the gain from the indexes in $I_1$ will be significantly smaller. So we distinguish when $p \ge 2$, when the function $t \mapsto t^{p/2}$ is convex so that we can exploit convexity, and the case $1 \le p < 2$, where actually the fact that $|x - \proj_d(x)|, |x-\bar T(x)| \le 2$ help us but instead of a Taylor expansion, using the mean value theorem. This gives 
    \begin{equation}
        \Delta_p(x) 
        \ge 
        \begin{cases}
            C \Delta_2(x),& \text{ for } 1 \le p < 2,\\ 
            C \delta^{p-2}\Delta_2(x),& \text{ for } p \ge 2.
        \end{cases}
    \end{equation}
    In in case 1, integrating over each $C_\varepsilon^i$ and up to changing the constants $C_1, C_2$, we obtain a lower bound for ${(\Delta W)}_1$ as 
    \begin{equation}\label{eq.case_2DW1}
        {(\Delta W)}_1 \ge 
        \begin{cases}
            C_1 \delta \varepsilon^2 - C_2 \varepsilon^3,& \text{ for } 1 \le p < 2,\\ 
            C_1 \delta^{p-1} \varepsilon^2 - C_2 \delta^{p-2}\varepsilon^3,& \text{ for } p \ge 2.  
        \end{cases}
    \end{equation}
    
    For $i \in I_2$ and $p \ge 2$, as before we have that 
    \begin{equation}\label{eq.case2_I2_dist_bound}
        \Delta_2(x) = 
        |x - \proj_d(x)|^2 - |x - y_\varepsilon|^2 \ge - \varepsilon^2,
    \end{equation}
    which can also be turned into a lower bound of $\Delta_p(x)$ by means of an application of the mean value theorem, finding some $d = \lambda|x - \proj_d(x)|^2 + (1-t)|x - y_\varepsilon|^2$ such that $\Delta_p(x) = d^{\frac{p}{2}-1}\Delta_2(x)$. As a result we have $\delta^{p-2} \le d \le C$ which gives us a bound as 
    \begin{equation}\label{eq.case2_I2_distp_bound}
        \Delta_p(x) \ge  
        \begin{cases}
           - C_3\delta^{p-2} \varepsilon^2,& \text{ if $1 \le p < 2$},\\ 
           - C_3\varepsilon^2,& \text{ if $p \ge 2$.}
        \end{cases}
    \end{equation}
    Recalling that $\bar \sigma(C^i_\varepsilon) = o(\varepsilon)$, this gives the following estimate on the loss in terms of transportation distance
    \begin{equation}\label{eq.case_2DW2}
        {(\Delta W)}_2 \ge 
        \begin{cases}
            - C_3\delta^{p-2} \varepsilon^2 o(\varepsilon),& \text{ for } 1 \le p < 2,\\ 
            - C_3\varepsilon^2 o(\varepsilon),& \text{ for } p \ge 2.  
        \end{cases}
    \end{equation}

    Finally, we can estimate the loss in transportation associated with only considering the mass outside the cylinder very easily as 
    \begin{equation}\label{eq.case_2DW3}
        |{(\Delta W)}_3| \le 
        \int_{C_{\delta,\varepsilon}} |x - y_\varepsilon|^p \dd \bar \sigma 
        \le 
        C_4{(\delta^2 + 4\varepsilon^2)}^{\frac{p}{2}}
        \bar \sigma(C_{\delta,\varepsilon}) 
        \le 
        C_4{(\delta^2 + 4\varepsilon^2)}^{\frac{p}{2}}
        \delta^{d-1}\varepsilon.
    \end{equation}
    
    To finish the proof, we insert the estimates~\eqref{eq.case_2DW1},\eqref{eq.case_2DW2},\eqref{eq.case_2DW3} into~\eqref{eq.case2_decomposition_trans_gain}. In the case $p \ge 2$, we obtain the following lower bound on the difference of transportation costs
    \begin{equation}\label{eq.case2_decomposition_trans_gain_p_ge2}
        W_p^p(\bar \sigma, \bar \nu) -  W_p^p(\bar \sigma, \nu')
        \ge 
        C_1 \delta^{p-1} \varepsilon^2 - C_2 \delta^{p-2}\varepsilon^3
        - C_3\varepsilon^2 o(\varepsilon)
        - C_4{(\delta^2 + 4\varepsilon^2)}^{\frac{p}{2}}
        \delta^{d-1}\varepsilon. 
    \end{equation}
    Since we can only possibly improve the transportation cost with the term $(\Delta W)_1$, the goal is to choose $\delta$ and $\varepsilon$ so that the exponent of the first term above is smaller than the remaining terms. 
    When $p \ge 2$, it is convenient to choose $\delta = \kappa \varepsilon$, for a constant $\kappa$ sufficiently large. First, taking $\kappa > 1$ makes the previous construction feasible as we needed $\delta \ge \varepsilon$, but with this choice the gain in transportation cost becomes 
    \[
        W_p^p(\bar \sigma, \bar \nu) -  W_p^p(\bar \sigma, \nu')
        \ge 
        C_1 \kappa^{p-1} \varepsilon^{p+1} - 
        C_2 \kappa^{p-2} \varepsilon^{p+1}- C_3\varepsilon^2 o(\varepsilon)
        - C_4{(\kappa^2 + 4)}^{\frac{p}{2}}\kappa^{d-1}\varepsilon^{p + d}. 
    \]
    Taking $\varepsilon$ small enough and $\kappa$ big enough, the right-hand side above can be made strictly positive, which gives a contraction to the entire construction and implies that $\Sigma$ cannot have loops in the case $p \ge 2$.

    In the regime $1 \le p < 2$, inserting~\eqref{eq.case_2DW1},\eqref{eq.case_2DW2},\eqref{eq.case_2DW3} into~\eqref{eq.case2_decomposition_trans_gain} gives 
    \begin{equation}\label{eq.case2_decomposition_trans_gain_1<p<2}
        W_p^p(\bar \sigma, \bar \nu) -  W_p^p(\bar \sigma, \nu')
        \ge 
        C_1 \delta \varepsilon^2 - C_2 \varepsilon^3 - 
        C_3\delta^{p-2} \varepsilon^2 o(\varepsilon) -
        C_4{(\delta^2 + 4\varepsilon^2)}^{\frac{p}{2}}
        \delta^{d-1}\varepsilon.
    \end{equation}
    Now, it is convenient to choose $\delta$ of the form $\delta = \varepsilon^\beta$,  for a good exponent $\beta > 0$ to be chosen. Then~\eqref{eq.case2_decomposition_trans_gain_1<p<2} becomes 
    \[
        W_p^p(\bar \sigma, \bar \nu) -  W_p^p(\bar \sigma, \nu')
        \ge 
        C_1 \varepsilon^{2+\beta} - C_2 \varepsilon^3 - 
        C_3\varepsilon^{(p-2)\beta} \varepsilon^3 \frac{o(\varepsilon)}{\varepsilon} -
        C_4\varepsilon^{(d-1+p)\beta + 1}\varepsilon.
    \]
    As a result, as long as the exponent of the first term is smaller than the others, we can make the right-hand side strictly positive by taking $\varepsilon$ small enough. This is equivalent to choosing $\beta$ such that 
    \begin{align*}
        2+\beta < 3,\quad 
        2+\beta \le \beta(p-2) + 3, \quad 
        2+\beta < (d-1+p)\beta + 1. 
    \end{align*}
    The first condition above is equivalent to $0 < \beta < 1$, the second is equivalent to $\beta \le \frac{1}{p-1}$, whose right-hand is always bigger than $1$ since $p<2$ so it is not restrictive, and finally the last condition is equivalent to $\frac{1}{d-2+p} < \alpha < 1$. These conditions can always be simultaneously satisfied for $d = 2$ and $p>1$ and for $d\ge 3$ and $1 \le p < 2$, our assumptions. As done previously, the result follows by making $\varepsilon$ small enough and performing the construction described above. 
    }
\end{proof}

\appendix

\section{Appendix: technical proofs of the concentration/blow-up argument}\label{appendix}
In this appendix we give the technical proofs of Lemma~\ref{lemma.localization} and Theorem~\ref{theorem.gamma_conv_Fn}, which are strongly inspired on the arguments from~\cite{chambolle2025one}. We recall that as throughout Section~\ref{sec.absense_loops} $\nu_\star$ is a fixed minimizer of the relaxed problem and $\alpha \eqdef \mathcal{L}(\nu_\star)$. {  We recall here that $\varrho_0$ is the fixed measure we approximate with a one-dimensional measure and that $\varrho_n,\sigma_n, \Sigma_n$ are obtained with the concentration/blow-up procedure described in the start of Section~\ref{subsec.blow-up}.} 
\begin{lemma}
    The localized measure $\nu_n$ solves the following minimization problem
    \begin{equation}\label{eq.localized_problem_appendix}
        \min \left\{
            W_p^p(\sigma_n, \nu'): 
            \substack{
                \displaystyle
                \text{ there is }
                \Sigma' \in \mathcal{A}_2 
                \text{ such that }\\  
                \\
                \displaystyle
                \nu' \in \mathscr{M}_+(\Sigma'), \ 
                \nu' \ge \alpha^{-1}\H^1\mres \Sigma',\\ 
                \\
                \displaystyle
                \Sigma_n \cup \Sigma' \text{ is connected,}\\ 
                \\
                \displaystyle
                \ \nu'(\overline{B_{1}(0)}) = 
                \nu_\star\left(\Sigmayrn\right)\\
            }           
        \right\}.
    \end{equation}
\end{lemma}
\begin{proof}
    Let $\gamma$ be the optimal transportation plan between $\varrho_0$ and $\nu_\star$. Recall the notation  
    \[
        \Sigmayrn \eqdef \Sigma \cap B_{r_n}(y_0) \text{ and }
        \Sigma_n \eqdef \Sigma \setminus B_{r_n}(y_0).
    \]
    By construction both sets are connected, and define the new transportation plan 
    \[
        \tilde \gamma \eqdef \gamma\mres \mathbb{R}^d\times\Sigma_n + \gamma',    
    \]
    where $\gamma'$ is optimal between $\varrho_n$ and $\nu'$. Then the new competitor $\tilde \nu \eqdef {(\pi_1)}_\sharp \tilde \gamma$ is such that $\mathcal{L}(\tilde\nu) \le \mathcal{L}(\nu_\star)$, and the optimality of $\nu_\star$ gives that
    \[
        \int_{\mathbb{R}^d\times \Sigma_n}|x-y|^p\dd \gamma + 
        \int_{\mathbb{R}^d\times \Sigmayrn}|x-y|^p\dd \gamma 
        \le 
        \int_{\mathbb{R}^d\times \Sigma_n}|x-y|^p\dd \gamma + 
        \int|x-y|^p\dd \gamma'
    \]
    Giving that $W_p^p(\varrho_n, \nu_n) \le W_p^p(\varrho_n, \nu')$ for all $\nu'$ admissible. 

    But we need to test the optimality of $\nu_n$ for the transport with initial measure given by $\sigma_n$. The latter was constructed to be a geodesic interpolation between $\varrho_n$ and $\nu_n$, see for instance~\cite[Theorem~5.27]{santambrogio2015optimal}. As such, it holds that 
    \begin{align*}
        W_p(\varrho_n, \sigma_n) + W_p(\sigma_n, \nu_n) 
        &= 
        W_p(\varrho_n, \nu_n)\\ 
        &\le 
        W_p(\varrho_n, \nu')
        \le 
        W_p(\varrho_n, \sigma_n) + W_p(\sigma_n, \nu'),
    \end{align*} 
    where above we have used the optimality of $\nu_n$ for the transport with $\varrho_n$ and the triangle inequality. Canceling the terms $W_p(\varrho_n, \sigma_n)$, the result follows.
\end{proof}

In the sequel, we prove Theorem~\ref{theorem.gamma_conv_Fn}. In fact, problem~\eqref{eq.localized_problem_appendix}, and consequently the functionals $F_n$ and $F$, have been modified from their counterparts in~\cite{chambolle2025one} in order to simplify the $\Gamma$-convergence result that follows. Whereas the formulation in~\cite{chambolle2025one} was chosen to be as general as possible; here we intend to show how we can facilitate greatly this proof by considering perturbations that are connected. 
\begin{theorem}
    The family ${\left(F_n\right)}_{n \in \mathbb{N}}$ converges to $F$ in the sense of $\Gamma$-convergence, for the topology of weak-$\star$ convergence of Radon measures. 
\end{theorem}
\begin{proof}
    Let us start with the $\Gamma-\liminf$, so given some measure $\nu'$ and a sequence ${\left(\nu'_n\right)}_{n \in \mathbb{N}}$ converging in the narrow topology to $\nu'$, and such that $\liminf_{n \to \infty} F_n(\nu'_n) < +\infty$, so we can assume that for each $n \in \mathbb{N}$ there is a set $\Sigma_n'$ such at most $2$ connected components such that 
    \[
        \Sigma_n' = \supp \nu'_n, \quad 
        \Sigma_n' \subset \overline{B_1(0)}, \quad
        \alpha \nu'_n \ge \H^1\mres \Sigma_n'.  
    \]

    Since $\Sigma_n'\subset \overline{B_1(0)}$, we can apply Blaschke's Theorem assuming that $\Sigma_n' \cvstrong{n \to \infty}{d_H} \Sigma'$, up to a not relabelled subsequence. The limit $\Sigma'$ also has at most $2$ connected components; and applying \Golab's Theorem to $\nu'$ restricted to each connected component it holds that 
    \[
        \alpha \nu' \ge \H^1\mres \Sigma' \text{ and } \nu' \in \mathscr{M}_+(\Sigma').    
    \]
    In addition, recall that by the construction from Lemma~\ref{lemma.noncut_property} 
    \[
        \frac{\Sigma_n - y_0}{r_n} \cap \partial B_1(0) 
        = \{y_{1,n}, y_{2,n}\}
        = 
        \frac{\Sigmayrn - y_0}{r_n} \cap \partial B_1(0).
    \]
    Since $\displaystyle \frac{\Sigmayrn - y_0}{r_n}$ converges to $[-\tau, \tau]$ we must have that $y_{i,n}\cvstrong{n \to \infty}{}{(-1)}^i\tau$ for $i = 1,2$. But since $\Sigma_n' \subset \overline{B_1(0)}$, the only way it is connected to $\displaystyle \frac{\Sigma_{n} - y_0}{r_n}$ is if it contains at least one of $y_{i,n}$, or both if it has two connected components. We then conclude that at least one of $-\tau, \tau$ belong to $\Sigma'$. 

    As a result, $\nu'$ is in the domain of $F$ and from the lower semi-continuity of the Wasserstein distance we get that 
    \begin{align*}
        F(\nu') 
        =
        W_p^p(\bar \sigma, \nu') 
        \le 
        \liminf_{n \to \infty}
        W_p^p(\bar \sigma_n, \nu'_n)
        =  
        \liminf_{n \to \infty}
        F_n(\nu'_n). 
    \end{align*}

    \par $\Gamma$-limsup: The strategy to prove the limsup is based on three steps: first we renormalize $\nu'$ to satisfy the mass constraint in $F_n$, which may break the condition $\alpha \nu_n'\ge \H^1\mres \Sigma'$, so we shrink the support to satisfy it again. Assuming that $\Sigma'$ has two connected components $\Sigma'_1,\Sigma'_2$, we translate the mass of each of their shrunk versions so that it is connected to $\displaystyle \frac{\Sigma_n - y_0}{r_n}$. Since some parts of the support may get out of $\overline{B_1(0)}$, we project the residual mass onto $\overline{B_1(0)}$.

Let us construct a recovery sequence ${(\nu_n')}_{n\in \mathbb{N}}$. By the constraint that $T_{y_0}\Sigma\cap \partial B_1(0) \subset \Sigma'$, the unit vectors $\pm\tau$ must be contained in each of the connected components $\Sigma_1',\Sigma_2'$. It is also possible that one of them is just a singleton $\pm\tau$ and only the other has positive length, or that $\Sigma'$ has only one connected component which contains both, but the following argument works for both cases with straightforward adaptations. By the Kuratowski (even Hausdorff) convergence  of $\Sigmayrn$
 towards $[-\tau,\tau]$, for each $i\in I$, there exists a sequence ${(y_{i,n})}_{n\in \mathbb{N}}$ such that $y_{i,n}\in \Sigmayrn$ for each $n\in \mathbb{N}$, and $y_{i,n}\to {(-1)}^{i}\tau$. We then define: 
\begin{equation*}
    a_n\eqdef 
    \frac{\nu_{\star}(\Sigmayrn)}{2r_n\theta_{1}(\nu_\star,y_0)},
    \quad\mbox{and}\quad   
    s_n\eqdef \max(1,a_{n}^{-1}), 
\end{equation*}
noting that $a_n\to 1$ and $s_n\to 1$, and we introduce the map $T_n$,
\[
    T_n(y) 
    \eqdef 
    (y + {(-1)}^{i+1}\tau)/s_n+y_{i,n}, \text{ if $y \in \Sigma'_i$}
\]
The map $T_n$ shrinks each connected component $\Sigma'_i$ and translates it to the corresponding $y_{i,n} \in \Sigmayrn$. It follows that 
\[
    \frac{\Sigma_n - y_0}{r_n} \cup T_n\left(\Sigma'\right)
\]
is connected, but not necessarily contained in $\overline{B_1}$; so we project it onto it and preserve connectedness. To perform this operation, let $\proj_{B_1}$ denote the projection onto the closed unit ball and define
\[ 
    \nu_n'\eqdef 
    {(\proj_{B_1}\circ T_n)}_{\sharp}(a_n\nu')
    \text{ and }
    \Sigma_n' 
    \eqdef 
    (\proj_{B_1}\circ T_n)(\Sigma').
\]
 
Let us check that $\nu_n'$ converges to $\nu'$ in the narrow topology. For $y\in \Sigma'_{i}$,
\[
    |y-T_{n}(y)|\le |y|(1-1/s_{n}) + |{(-1)}^{i+1}/s_n-y_{i,n}|\xrightarrow[n\to +\infty]{} 0.
\]
By the dominated convergence theorem, we get that for any $\phi\in \mathscr{C}_b(\mathbb{R}^d)$,

\[ 
    \int\phi\dd \nu_n' 
    = 
    a_n\int_{\Sigma'}\phi\left(\proj_{B_1}\circ T_n(y)\right)\dd \nu'(y) 
    \xrightarrow[n\to +\infty]{}  
    \int_{\Sigma'}\phi\left(y\right)\dd \nu'(y)
\]
so that $\nu_n'\xrightharpoonup[n\to +\infty]{}\nu'$ in the narrow topology.

Let us now check the constraints in $F_n$. From the properties of image measures, we see that the mass of $\nu_n'$ is concentrated in $\Sigma_n' \subset \overline{B_1(0)}$ which is such that 
\[
    \frac{\Sigma_n - y_0}{r_n}\cup \Sigma'_n,
\]
is connected by the previous arguments, and we also have
\[
    \nu_n'(\overline{B_{1}(0)}) 
    = 
    a_n\nu'(\mathbb{R}^{d})
    = 
    \frac{\nu_{\star}(\Sigmayrn)}{r_n}
\]
so that $\nu_n'$ has the mass prescribed by $F_n$.

It only remains to show that is satisfies the density constraints, take any non-negative $\phi\in \mathscr{C}_b(\mathbb{R}^d)$,
\begin{align*}
    \alpha\int_{\mathbb{R}^d}\phi\dd \nu_n' 
    &= \alpha a_n\int_{\Sigma'}
    \phi\left(\proj_{B_1}\circ T_n(y)\right)\dd \nu'(y)\\
    &\ge a_n
    \int_{\Sigma'}\phi\left(\proj_{B_1}\circ T_n(y)\right)\dd \H^{1}(y)\\
    &= a_n s_{n}\int_{\Sigma_n'}
    \phi\left(\proj_{B_1}\circ T_n(y')\right)\dd \H^{1}(y')
    \ge \int_{\Sigma'_{n}} \phi \dd \H^{1}. 
\end{align*}
It follows that $\alpha \nu_n'\geq \H^1\mres \Sigma'_{n}$ and we conclude that $F_n(\nu_n) < \infty$, for all $n \in \mathbb{N}$.

By the continuity of the Wasserstein distance with
respect to the narrow convergence (provided the measures are supported in some common compact set), we have that:
    \[
        F_n(\nu_n') \xrightarrow[n \to \infty]{} F(\nu').
    \]
The $\Gamma$-convergence follows.

\end{proof}


\bibliographystyle{plain}
\bibliography{references.bib}


\end{document}